\documentclass[reqno]{amsart}
\usepackage{amssymb}

\usepackage[colorlinks=true]{hyperref}
\hypersetup{urlcolor=blue, citecolor=red}

\newtheorem{Theorem}{Theorem}[section]
\newtheorem{Corollary}[Theorem]{Corollary}

\newtheorem{Proposition}[Theorem]{Proposition}

\numberwithin{equation}{section}

\def\C {\mathbb C}

\def\R {\mathbb R}
\def\Z {\mathbb Z}

\newcommand{\<}{\langle}
\renewcommand{\>}{\rangle}
\renewcommand{\(}{\left(}
\renewcommand{\)}{\right)}
\newcommand{\lesim}{\lesssim}
\renewcommand{\Re}{\operatorname{Re}}
\renewcommand{\Im}{\operatorname{Im}}

\newcommand{\supp}{\operatorname{supp}}

\newcommand{\p}{\partial}

\newcommand{\id}{\operatorname{Id}}

\begin{document}
\title[Inverse problems for polyharmonic operator]{Inverse problems for the perturbed polyharmonic operator with coefficients in Sobolev spaces with non-positive order}

\author[Yernat M. Assylbekov]{Yernat M. Assylbekov}
\address{Department of Mathematics, University of Washington, Seattle, WA 98195-4350, USA}
\email{y\_assylbekov@yahoo.com}

\maketitle

\begin{abstract}
We show that the knowledge of the Dirichlet-to-Neumann map on the boundary of a bounded open set in $\R^n$, $n\ge 3$, for the perturbed polyharmonic operator $(-\Delta)^m+A\cdot D+q$, $m\ge 2$, with $n>m$, $A\in W^{-\frac{m-2}{2},\frac{2n}{m}}$ and $q\in W^{-\frac{m}{2}+\delta,\frac{2n}{m}}$, with $0<\delta<1/2$, determines the potentials $A$ and $q$ in the set uniquely. The proof is based on a Carleman estimate with linear weights and with a gain of two derivatives and on the property of products of functions in Sobolev spaces.
\end{abstract}

\section{Introduction}
Let $\Omega\subset \R^n$, $n\ge 3$, be a bounded open set with $C^\infty$ boundary. Consider the  polyharmonic operator $(-\Delta)^m$, where $m\ge 1$ is an integer. The operator $(-\Delta)^m$ is positive and self-adjoint on $L^2(\Omega)$ with domain $H^{2m}(\Omega)\cap H^m_0(\Omega)$, where
$$
H^m_0(\Omega)=\{ 
u\in H^m(\Omega): \gamma u=0\}.
$$
This operator can be obtained as the Friedrichs extension starting from the space of test functions. This fact can be found, for example, in \cite{Gr}. Here and in what follows, $\gamma$ is the Dirichlet trace operator
$$
\gamma:H^m(\Omega)\to \prod_{j=0}^{m-1}H^{m-j-1/2}(\p \Omega),\quad \gamma u=(u|_{\p\Omega},\p_\nu u|_{\p\Omega},\dots,\p_\nu^{m-1}u|_{\p\Omega}),
$$
where $\nu$ is the unit outer normal to the boundary $\p \Omega$, and $H^s(\Omega)$ and $H^s(\p\Omega)$ are the standard $L^2$-based Sobolev spaces on $\Omega$ and its boundary $\p\Omega$ for $s\in\R$.
\medskip

Throughout the paper we shall assume that $n>m$.
\medskip

Let $A\in W^{-\frac{m-2}{2},\frac{2n}{m}}(\R^n,\C^n)\cap \mathcal E'(\overline{\Omega},\C^n)$ and $q\in W^{-\frac{m}{2},\frac{2n}{m}}(\R^n,\C)\cap \mathcal E'(\overline{\Omega},\C)$, where $\mathcal E'(\overline{\Omega})=\{u\in \mathcal D'(\R^n):\supp (u)\subseteq\overline \Omega\}$ and $W^{s,p}(\R^n)$ is the standard $L^p$-based Sobolev space on $\R^n$, $s\in\R$ and $1<p<\infty$, which is defined by the Bessel potential operator. Thus $W^{s,p}(\R^n)$ is the space of all distributions $u$ on $\R^n$ such that $J_{-s}u\in L^p(\R^n)$, where $J_s$ is the operator defined as
$$
u\mapsto ((1+|\cdot|^2)^{-s/2}\widehat u(\cdot))^\vee.
$$
In the case $s\ge 0$ integer, $W^{s,p}(\R^n)$ coincides with the space of all functions whose all derivatives of order less or equal to $s$ is in $L^p(\R^n)$. The reader is referred to \cite{Trieb} for properties of these spaces.
\medskip

Before stating the problem, we consider the bilinear forms $B_A$ and $b_q$ on $H^m(\Omega)$ which are defined by
$$
B_A(u,v)=\<A,\tilde v\,D\tilde u\>_{\R^n},\quad b_q(u,v)=\<q,\tilde u\,\tilde v\>_{\R^n},\quad u,v\in H^m(\Omega),
$$
where $\<\cdot,\cdot\>_{\R^n}$ is the distributional duality on $\R^n$, and $\tilde u,\tilde v \in H^m(\R^n)$ are extensions of $u$ and $v$, respectively. In Appendix~\ref{appendix a}, we show that these definitions are well-defined (i.e. independent of the choice of extensions $\tilde u,\tilde v$). Using a property on multiplication of functions in Sobolev spaces, we show that the forms $B_A$ and $b_q$ are bounded on $H^m(\Omega)$; see Proposition~\ref{m_q and D_A are bounded maps on Omega}.
\medskip

Consider the operator $D_A$ which is formally $A\cdot D$, where $D_j=-i\p_{x_j}$, and the operator $m_q$ of multiplication by $q$. To be precise, for $u\in H^m(\Omega)$, $D_A(u)$ and $m_q(u)$ are defined as
$$
\<D_A(u),\psi\>_\Omega=B_A(u,\psi)\text{ and }\<m_q(u),\psi\>_\Omega=b_q(u,\psi),\quad \psi\in C^\infty_0(\Omega).
$$
Then the operators $D_A$ and $m_q$ are bounded $H^m(\Omega)\to H^{-m}(\Omega)$ (see Corollary~\ref{m_q and D_A are well-defined on HOmega}), and hence, standard arguments show that the operator
$$
\mathcal L_{A,q}=(-\Delta)^m+D_A+m_q:H^m_0(\Omega)\to H^{-m}(\Omega)=(H^m_0(\Omega))',
$$
is Fredholm operator with zero index; see Appendix~\ref{appendix b}.
\medskip


For $f=(f_0,\dots,f_{m-1})\in \prod_{j=0}^{m-1}H^{m-j-1/2}(\p\Omega)$, consider the Dirichlet problem
\begin{equation}\label{DP}
\begin{aligned}
\mathcal L_{A,q}u&=0\quad\text{in}\quad\Omega,\\
\gamma u&=f\quad\text{on}\quad\p\Omega.
\end{aligned}
\end{equation}
If $0$ is not in the spectrum of $\mathcal L_{A,q}$, the Dirichlet problem \eqref{DP} has a unique solution $u\in H^m(\Omega)$. We define the Dirichlet-to-Neumann map $\mathcal N_{A,q}$ as follows
\begin{equation}\label{N_A,q}
\<\mathcal N_{A,q}f,\overline h\>_{\p\Omega}:=\sum_{|\alpha|=m}\frac{m!}{\alpha!}(D^\alpha u,\overline{D^\alpha v}_h)_{L^2(\Omega)}+B_A(u,\overline v_h)+b_q(u,\overline v_h),
\end{equation}
where $h=(h_0,\dots,h_{m-1})\in \prod_{j=0}^{m-1}H^{m-j-1/2}(\p\Omega)$, and $v_h\in H^m(\Omega)$ is an extension of $h$, that is $\gamma v_h=h$.
It is shown in Appendix~\ref{appendix b} that $\mathcal N_{A,q}$ is a well-defined (i.e. independent of the choice of $v_h$) bounded operator
$$
\mathcal N_{A,q}:\prod_{j=0}^{m-1}H^{m-j-1/2}(\p\Omega)\to\(\prod_{j=0}^{m-1}H^{m-j-1/2}(\p\Omega)\)'=\prod_{j=0}^{m-1}H^{-m+j+1/2}(\p\Omega).
$$
The inverse boundary problem for the perturbed polyharmonic operator $\mathcal L_{A,q}$ is to determine $A$ and $q$ in $\Omega$ from the knowledge of the Dirichlet-to-Neumann map $\mathcal N_{A,q}$.\medskip 

When $m=1$ the operator $L_{A,q}$ is the first order perturbation of the Laplacian and $N_{A,q}u$ is formally given by $N_{A,q}f=(\p_\nu u+i(A\cdot \nu)u)|_{\p\Omega}$, where $u$ is an $H^m(\Omega)$ solution to the equation $L_{A,q}u=0$. It was shown in \cite{Sun} that in this case there is an obstruction to uniqueness in this problem given by the following gauge equivalence of the set of the Cauchy data: if $\psi\in W^{1,\infty}$ in a neighbourhood of $\overline\Omega$ and $\psi|_{\p\Omega}=0$, then $C_{A,q}=C_{A+\nabla \psi,q}$; see also \cite[Lemma~3.1]{KU}. Hence, given $C_{A,q}$, we may only hope to recover the magnetic field $dA$ and electric potential $q$. Here and in what follows the magnetic field $dA$ is defined by
$$
dA=\sum_{1\le j,k,\le n}\(\p_{x_j}A_k-\p_{x_k}A_j\)dx_j\wedge dx_k.
$$
Due to the lack smoothness of $A$, this definition is in the sense of distributions.
\medskip

Starting with the paper of Sun \cite{Sun}, inverse boundary value problems for the magnetic Schr\"odinger operators have been extensively studied. It was shown in \cite{Sun} that the hope mentioned above is justified provided that $A\in W^{2,\infty}$, $q\in L^\infty$ and $dA$ satisfies a smallness condition. The smallness condition was removed in \cite{NSU} for $C^\infty$ magnetic and electric potentials, and also for compactly supported $C^2$ magnetic potentials and essentially bounded electric potentials. The regularity assumption on magnetic potentials were subsequently weakened to $C^1$ in  \cite{Tol}, and then to Dini continuous in \cite{S}. All these results were obtained under the assumption that zero is not a Dirichlet eigenvalue for the magnetic Schr\"odinger operator in $\Omega$. There are two best result by now. One is due to Krupchyk and Uhlmann \cite{KU}, where they prove uniqueness under the assumption that magnetic and electric potentials are of class $L^\infty$. Another is due to Haberman \cite{Hab2} in dimension $n=3$, where the uniqueness is shown for the case when $q\in W^{-1,3}$ and $A\in W^{s,3}$ for some $s>0$ with certain smallness condition.
\medskip

It was shown in \cite{KLU} that the obstruction to uniqueness coming from the gauge equivalence when $m=1$ can be eliminated by considering operators of higher order. More precisely, they show that for $m\ge 2$ the set of Cauchy data $C_{A,q}$ determines $A$ and $q$ uniquely provided that $A\in W^{1,\infty}(\Omega,\C^n)\cap\mathcal E'(\overline\Omega,\C^n)$ and $q\in L^\infty(\Omega)$. They also show that the uniqueness result holds without the assumption $A=0$ on $\p\Omega$ but for $C^\infty$ magnetic and electric potentials. This is also true for $A\in W^{1,\infty}(\Omega,\C^n)$ and $q\in L^\infty(\Omega)$ when the boundary of the domain $\Omega$ is connected.
\medskip

The purpose of this paper is to relax the regularity assumption on $A$ from $W^{1,\infty}$ to $W^{-\frac{m-2}{2},\frac{2n}{m}}$ class and $q$ from $L^\infty$ to $W^{-\frac{m}{2}+\delta,\frac{2n}{m}}$, $0<\delta<1/2$, for the perturbed polyharmonic operator $\mathcal L_{A,q}$ with $m\ge2$. Therefore, throughout the paper we assume that $m\ge2$. Our main result is as follows.
\begin{Theorem}\label{main th}
Let $\Omega\subset \R^n$, $n\ge 3$, be a bounded open set with $C^\infty$ boundary, and let $m\ge 2$ be an integer such that $n>m$. Suppose that $A_1,A_2\in W^{-\frac{m-2}{2},\frac{2n}{m}}(\R^n,\C^n)\cap\mathcal E'(\overline \Omega)$ and $q_1,q_2\in W^{-\frac{m}{2}+\delta,\frac{2n}{m}}(\R^n,\C)\cap\mathcal E'(\overline \Omega)$, for some $0<\delta<1/2$, are such that $0$ is not in the spectrums of $\mathcal L_{A_1,q_1}$ and $\mathcal L_{A_2,q_2}$. If $\mathcal N_{A_1,q_1}=\mathcal N_{A_2,q_2}$, then $A_1=A_2$ and $q_1=q_2$.
\end{Theorem}

The assumption $n>m$ is related to the dual space of $W^{s,\frac{2n}{m}}$ and the estimate on products of functions in different Sobolev spaces. It seems to the author that the techniques of the present paper can be adopted to the case $n\le m$ by changing regularity assumptions. We hope to consider this problem in future work.
\medskip

The key ingredient in the proof of Theorem \ref{main th} is a construction of  complex geometric optics solutions for the operator $A\in W^{-\frac{m-2}{2},\frac{2n}{m}}(\R^n,\C^n)\cap \mathcal E'(\overline{\Omega},\C^n)$ and $q\in W^{-\frac{m}{2},\frac{2n}{m}}(\R^n,\C)\cap \mathcal E'(\overline{\Omega},\C)$. For this, we use the method of Carleman estimates which is based on the corresponding Carleman estimate for the Laplacian, with a gain of two derivatives, due to Salo and Tzou \cite{STz}. Another important tool in our proof is the property of products of functions in Sobolev spaces \cite{RS}. This was used in the paper of Brown and Torres \cite{BrT}. The idea of constructing such solutions for the Schr\"odinger operator goes back to the fundamental paper due to Sylvester and Uhlmann \cite{SyU}. Such solutions were first introduced in \cite{Faddeev} in the setting of quantum inverse scattering problem.
\medskip

The inverse boundary value problem of the recovery of a zeroth order perturbation of the biharmonic operator, that is when $m=2$, has been studied by Isakov \cite{Isakov}, where uniqueness result was obtained, similarly to the case of the Schr\"odinger operator. In \cite{Ikeh}, the uniqueness result was extended to $q\in L^{n/2}(\Omega)$, $n>4$ by Ikehata. These results were extended for zeroth order perturbation of the polyharmonic operator with $q\in L^{n/2m}$, $n>2m$ by Krupchyk and Uhlmann \cite{KU2}. In the case $m=1$, that is for zeroth order perturbation of the Schr\"odinger operator, global uniqueness result was established by Lavine and Nachman \cite{LN} for $q\in L^{n/2}(\Omega)$, following an earlier result of Chanillo \cite{Ch} for $q\in L^{n/2+\varepsilon}(\Omega)$, $\varepsilon > 0$ and Novikov \cite{N} for $q\in L^\infty(\Omega)$.
\medskip

Higher ordered polyharmonic operators occur in the areas of physics and geometry such as the study of the Kirchhoff plate equation in the theory of elasticity, and the study of the Paneitz-Branson operator in conformal geometry; for more details see monograph \cite{GGS}.
\medskip

We would like to remark that the problem considered in this paper can be considered as generalization of the Calder\'on's inverse conductivity problem \cite{Cal}, known also as electrical impedance tomography, for which the reduction of regularity have been studied extensively. In the fundamental paper by Sylvester and Uhlmann \cite{SyU} it was shown that $C^2$ conductivities can be uniquely determined from boundary measurements. The regularity assumptions were weakened to $C^{3/2 + \varepsilon}$ conductivities by Brown \cite{Br}, and corresponding result for $C^{3/2}$ conductivities was obtained by P\"aiv\"arinta, Panchenko and Uhlmann \cite{PPU}. Uniqueness result for $C^{1+\varepsilon}$ conormal conductivities was shown by Greenleaf, Lassas and Uhlmann \cite{GLU}. There is a recent work by Haberman and Tataru \cite{HT} which gives a uniqueness result for Calder\'on’s problem with $C^1$ conductivities and with Lipschitz continuous conductivities, which are close to the identity in a suitable sense. Very recent work of Caro and Rogers \cite{CR} shows that Lipschitz conductivities can be determined from the Dirichlet-to-Neumann map. Finally, Haberman \cite{Hab} gives uniquess results for conductivities with unbounded gradient. In particular, uniqueness for conductivities in $W^{1,n}(\overline{\Omega})$ with $n=3,4$ is obtained.
\medskip

The structure of the paper is as follows. Section~\ref{CGOs} is devoted to the construction of  complex geometric optics solutions for the perturbed polyharmonic operator $\mathcal L_{A,q}$ with $A\in W^{-\frac{m-2}{2},\frac{2n}{m}}(\R^n,\C^n)\cap \mathcal E'(\overline{\Omega},\C^n)$ and $q\in W^{-\frac{m}{2},\frac{2n}{m}}(\R^n,\C)\cap \mathcal E'(\overline{\Omega},\C)$. Then the proof of Theorem~\ref{main th} is given in Section~\ref{Proof}. In Appendix~\ref{appendix a}, we study mapping properties of the operators $D_A$ and $m_q$. Finally, Appendix~\ref{appendix b} is devoted to the well-posedness of the Dirichlet problem for $\mathcal L_{A,q}$ with $A\in W^{-\frac{m-2}{2},\frac{2n}{m}}(\R^n,\C^n)\cap \mathcal E'(\overline{\Omega},\C^n)$ and $q\in W^{-\frac{m}{2},\frac{2n}{m}}(\R^n,\C)\cap \mathcal E'(\overline{\Omega},\C)$.

\section{Carleman estimates and Complex geometric optics solutions}\label{CGOs}
In this section we construct the complex geometric optics solutions for the equation $\mathcal{L}_{A,q}u=0$ in $\Omega$ with $A\in W^{-\frac{m-2}{2},\frac{2n}{m}}(\R^n,\C^n)\cap \mathcal E'(\overline{\Omega},\C^n)$ and $q\in W^{-\frac{m}{2},\frac{2n}{m}}(\R^n,\C)\cap \mathcal E'(\overline{\Omega},\C)$, $m\ge 2$. When constructing such solutions, we shall first derive Carleman estimates for the operator $\mathcal{L}_{A,q}$.

We start by recalling the Carleman estimate for the semiclassical Laplace operator $-h^2\Delta$ with a gain of two derivatives, established in \cite{STz}. Let $\widetilde \Omega$ be an open set in $\R^n$ such that $\Omega\subset\subset \widetilde \Omega$ and let $\varphi\in C^\infty(\widetilde \Omega,\R)$. Consider the conjugated operator
$$
P_\varphi=e^{\varphi/h}(-h^2\Delta)e^{-\varphi/h},
$$
and its semiclassical principal symbol
$$
p_\varphi(x,\xi)=\xi^2+2i\nabla\varphi\cdot\xi-|\nabla\varphi|^2,\quad x\in\widetilde\Omega,\quad \xi\in\R^n.
$$
Following \cite{KSU}, we say that $\varphi$ is a limiting Carleman weight for $-h^2\Delta$ in $\widetilde\Omega$, if $\nabla \varphi\neq 0$ in $\widetilde \Omega$ and the Poisson bracket of $\text{Re}\,p_\varphi$ and $\text{Im}\,p_\varphi$ satisfies
$$
\{\text{Re}\,p_\varphi,\text{Im}\,p_\varphi\}(x,\xi)=0\quad\text{when}\quad p_\varphi(x,\xi)=0,\quad (x,\xi)\in\widetilde\Omega\times\R^n.
$$
In this paper we shall consider only the linear Carleman weights $\varphi(x)=\alpha\cdot x$, $\alpha\in\R^n$, $|\alpha|=1$.

In what follows we consider the semiclassical norm on the standard Sobolev space $H^s(\R^n)$, $s\in\R$,
$$
\|u\|_{H^s_{\rm{scl}}(\R^n)}=\|\<hD\>^s u\|_{L^2(\R^n)},\quad \<\xi\>=(1+|\xi|^2)^{1/2}.
$$

Our starting point is the following Carleman estimate for the semiclassical Laplace operator $-h^2\Delta$ with a gain of two derivatives, which is due to Salo and Tzou \cite{STz}.
\begin{Proposition}\label{Carleman for -h^2Delta}
Let $\varphi$ be a limiting Carleman weight for $-h^2\Delta$ in $\widetilde\Omega$, and let $\varphi_\varepsilon=\varphi+\frac{h}{2\varepsilon}\varphi^2$. Then for $0<h\ll\varepsilon\ll 1$ and $s\in \R$, we have
$$
\frac{h}{\sqrt{\varepsilon}}\|u\|_{H^{s+2}_{\rm{scl}}(\R^n)}\le C\|e^{\varphi_\varepsilon/h}(-h^2\Delta)e^{-\varphi_\varepsilon/h}u\|_{H^{s}_{\rm{scl}}(\R^n)},\quad C>0,
$$
for all $u\in C_0^\infty(\Omega)$.
\end{Proposition}

Next, we state theorem on products of functions in Sobolev spaces. This result is well-known, see Theorem~2 in \cite[Subsection 4.4.4]{RS}.
\begin{Proposition}\label{product of Sobolev spaces}
Let $1<p,q<\infty$ and $0<s_1\le s_2<n\min(1/p,1/q)$. Then $W^{s_1,p}(\R^n)\cdot W^{s_2,q}(\R^n)\hookrightarrow W^{s_1,r}(\R^n)$ where $1/r=1/p+1/q-s_2/n$.
\end{Proposition}

Now we shall derive Carleman estimate for the perturbed operator $\mathcal L_{A,q}$ with $A\in W^{-\frac{m-2}{2},\frac{2n}{m}}(\R^n,\C^n)\cap \mathcal E'(\overline{\Omega},\C^n)$ and $q\in W^{-\frac{m}{2},\frac{2n}{m}}(\R^n,\C)\cap \mathcal E'(\overline{\Omega},\C)$. To that end we shall iterate ($m$ times) inequality in Proposition~\ref{Carleman for -h^2Delta} and use it with $s=-m$, and with fixed $\varepsilon>0$ being sufficiently small, that is independent of $h$. We have the following result.
\begin{Proposition}\label{Carleman est L_{A,q} prop}
Let $\varphi$ be a limiting Carleman weight for $-h^2\Delta$ in $\widetilde\Omega$, and suppose that $A\in W^{-\frac{m-2}{2},\frac{2n}{m}}(\R^n,\C^n)\cap \mathcal E'(\overline{\Omega},\C^n)$ and $q\in W^{-\frac{m}{2},\frac{2n}{m}}(\R^n,\C)\cap \mathcal E'(\overline{\Omega},\C)$. Then for $0<h\ll 1$, we have
\begin{equation}\label{Carleman est L_{A,q}}
\|u\|_{H^m_{\rm{scl}}(\R^n)}\lesim \frac{1}{h^m}\|e^{\varphi/h}(h^{2m}\mathcal{L}_{A,q})e^{-\varphi/h}u\|_{H^{-m}_{\rm{scl}}(\R^n)},
\end{equation}
for all $u\in C^\infty_0(\Omega)$.
\end{Proposition}
\begin{proof}
Iterating the Carleman estimate in Proposition~\ref{Carleman for -h^2Delta} $m$ times, $m\ge2$, we get the following Carleman estimate for the polyharmonic operator,
$$
\frac{h^m}{\varepsilon^{m/2}}\|u\|_{H^{s+2m}_{\rm{scl}}(\R^n)}\le C\|e^{\varphi_\varepsilon/h}(-h^2\Delta)^m e^{-\varphi_\varepsilon/h}u\|_{H^{s}_{\rm{scl}}(\R^n)},
$$
for all $u\in C^\infty_0(\Omega)$, $s\in\R$ and $0<h\ll \varepsilon\ll 1$. We shall use this estimate with $s=-m$, and with fixed $\varepsilon>0$ being sufficiently small but independent of $h$:
\begin{equation}\label{Carleman for (-h^2Delta)^m}
\frac{h^m}{\varepsilon^{m/2}}\|u\|_{H^{m}_{\rm{scl}}(\R^n)}\le C\|e^{\varphi_\varepsilon/h}(-h^2\Delta)^m e^{-\varphi_\varepsilon/h}u\|_{H^{-m}_{\rm{scl}}(\R^n)},
\end{equation}
for all $u\in C^\infty_0(\Omega)$ and $0<h\ll \varepsilon\ll 1$.
\medskip

In order to prove the proposition it will be convenient to use the following characterization of the semiclassical norm in the Sobolev space $H^{-m}(\R^n)$
\begin{equation}\label{H^-m characterization}
\|v\|_{H^{-m}_{\rm scl}(\R^n)}=\sup_{0\neq \psi\in C^\infty_0(\R^n)}\frac{|\<v,\psi\>_{\R^n}|}{\|\psi\|_{H^{m}_{\rm scl}(\R^n)}},
\end{equation}
where $\<\cdot,\cdot\>_{\R^n}$ is the distribution duality on $\R^n$.
\medskip

Let $\varphi_\varepsilon=\varphi+\frac{h}{2\varepsilon}\varphi^2$ be the convexified weight with $\varepsilon>0$ such that $0<h\ll \varepsilon\ll 1$, and let $u\in C^\infty_0(\Omega)$. Then for all $0\neq \psi\in C^\infty_0(\R^n)$, by duality of the spaces $W^{-\frac{m}{2},\frac{2n}{m}}(\R^n,\C)$ and $W^{\frac{m}{2},\frac{2n}{2n-m}}(\R^n,\C)$ and by Proposition~\ref{product of Sobolev spaces}, we have
\begin{align*}
|\<e^{\varphi_\varepsilon/h}h^{2m}m_q(e^{-\varphi_\varepsilon/h}u),\psi\>_{\R^n}|&\le Ch^{2m}\|q\|_{W^{-\frac{m}{2},\frac{2n}{m}}(\R^n)}\|u\psi\|_{{W^{\frac m2,\frac{2n}{2n-m}}}(\R^n)}\\
&\le Ch^{2m}\|q\|_{W^{-\frac{m}{2},\frac{2n}{m}}(\R^n)}\|u\|_{H^{\frac m2}(\R^n)}\|\psi\|_{H^{\frac m2}(\R^n)}\\
&\le Ch^m\|q\|_{W^{-\frac{m}{2},\frac{2n}{m}}(\R^n)}\|u\|_{H^{\frac m2}_{\rm scl}(\R^n)}\|\psi\|_{H^{\frac m2}_{\rm{scl}}(\R^n)}\\
&\le Ch^m\|q\|_{W^{-\frac{m}{2},\frac{2n}{m}}(\R^n)}\|u\|_{H^{m}_{\rm scl}(\R^n)}\|\psi\|_{H^{m}_{\rm{scl}}(\R^n)}.
\end{align*}
Therefore, by \eqref{H^-m characterization}, we obtain
\begin{equation}\label{q part}
\|e^{\varphi_\varepsilon/h}h^{2m}m_q(e^{-\varphi_\varepsilon/h}u)\|_{H^{-m}_{\rm scl}(\R^n)}\le Ch^m\|q\|_{W^{-\frac{m}{2},\frac{2n}{m}}(\R^n)}\|u\|_{H^m_{\rm scl}(\R^n)}.
\end{equation}
For all $0\neq \psi\in C^\infty_0(\Omega)$, we can show
\begin{align*}
|\<e^{\varphi_\varepsilon/h}&h^{2m}D_A(e^{-\varphi_\varepsilon/h}u),\psi\>_{\R^n}|=|\<h^{2m}A,e^{\varphi_\varepsilon/h}\psi D(e^{-\varphi_\varepsilon/h}u)\>_{\R^n}|\\
&\le |\<h^{2m-1}A,e^{\varphi_\varepsilon/h}\psi [-u(1+h\varphi/\varepsilon)D\varphi+hDu]\>_{\R^n}|\\
&\le Ch^{2m-1}\|A\|_{W^{-\frac{m-2}{2},\frac{2n}{m}}(\R^n)}\|-u(1+h\varphi/\varepsilon)D\varphi\psi+hDu\psi\|_{W^{\frac{m-2}{2},\frac{2n}{2n-m}}(\R^n)}.
\end{align*}
In the last step we used duality between $W^{-\frac{m-2}{2},\frac{2n}{m}}(\R^n,\C^n)$ and $W^{\frac{m-2}{2},\frac{2n}{2n-m}}(\R^n,\C^n)$. In the case $m\ge 3$, we use Proposition~\ref{product of Sobolev spaces} with $p=q=2$, $s_1=\frac{m-2}{2}$ and $s_2=\frac{m}{2}$ (this is exactly the moment where we need the stronger assumption $m<n$ rather than $m<2n$), to get
\begin{align*}
\|-u(1+h\varphi/\varepsilon)D\varphi\psi&+hDu\psi\|_{W^{\frac{m-2}{2},\frac{2n}{2n-m}}(\R^n)}\\
&\le C\|-u(1+h\varphi/\varepsilon)D\varphi+hDu\|_{H^{\frac{m-2}{2}}(\R^n)}\|\psi\|_{H^{\frac{m}{2}}(\R^n)}\\
&\le \frac{C}{h^{m-1}}\|u\|_{H^{\frac{m}{2}}_{\rm{scl}}(\R^n)}\|\psi\|_{H^{\frac{m}{2}}_{\rm{scl}}(\R^n)}\\
&\le  \frac{C}{h^{m-1}}\|u\|_{H^{m}_{\rm{scl}}(\R^n)}\|\psi\|_{H^{m}_{\rm{scl}}(\R^n)},
\end{align*}
for some constant $C>0$ depending only on $\varphi$. When $m=2$, we use H\"older's inequality and Sobolev embedding $H^{1}(\R^n)\subset L^{\frac{2n}{n-2}}(\R^n)$ (see \cite[Chapter~13, Proposition~6.4]{Taylor3}), and obtain
\begin{align*}
\|-u(1+h\varphi/\varepsilon)D\varphi\psi&+hDu\psi\|_{L^{\frac{2n}{2n-2}}(\R^n)}\\
&\le C\|-u(1+h\varphi/\varepsilon)D\varphi+hDu\|_{L^{\frac{2n}{n}}(\R^n)}\|\psi\|_{L^{\frac{2n}{n-2}}(\R^n)}\\
&\le C\|-u(1+h\varphi/\varepsilon)D\varphi+hDu\|_{L^2(\R^n)}\|\psi\|_{H^{1}(\R^n)}\\
&\le \frac{C}{h}\|u\|_{H^{1}_{\rm{scl}}(\R^n)}\|\psi\|_{H^{1}_{\rm{scl}}(\R^n)}\\
&\le  \frac{C}{h}\|u\|_{H^{2}_{\rm{scl}}(\R^n)}\|\psi\|_{H^{2}_{\rm{scl}}(\R^n)},
\end{align*}
for some constant $C>0$ depending only on $\varphi$. Therefore, for $m\ge 2$, we get
$$
|\<e^{\varphi_\varepsilon/h}h^{2m}D_A(e^{-\varphi_\varepsilon/h}u),\psi\>_{\R^n}|\le C h^m\|A\|_{W^{-\frac{m-2}{2},\frac{2n}{m}}(\R^n)}\|u\|_{H^{m}_{\rm{scl}}(\R^n)}\|\psi\|_{H^{m}_{\rm{scl}}(\R^n)}.
$$
Hence, by \eqref{H^-m characterization}, we obtain
$$
\|e^{\varphi_\varepsilon/h}h^{2m}D_A(e^{-\varphi_\varepsilon/h}u)\|_{H^{-m}_{\rm{scl}}(\R^n)}\le Ch^m\|A\|_{W^{-\frac{m-2}{2},\frac{2n}{m}}(\R^n)}\|u\|_{H^{m}_{\rm{scl}}(\R^n)}.
$$
Combining these estimates with \eqref{Carleman for (-h^2Delta)^m} and \eqref{q part} we get that for small enough $h>0$
$$
\|u\|_{H^m_{\rm{scl}}(\R^n)}\lesssim \frac{1}{h^m}\|e^{\varphi_\varepsilon/h}(h^{2m}\mathcal L_{A,q})e^{-\varphi_\varepsilon/h}u\|_{H^{-m}_{\rm{scl}}(\R^n)}.
$$
Using that
$$
e^{-\varphi_\varepsilon/h}u=e^{-\varphi/h}e^{-\varphi^2/{2\varepsilon}}u
$$
we obtain \eqref{Carleman est L_{A,q}}.
\end{proof}

Let $\varphi\in C^\infty(\widetilde\Omega,\R)$ be a limiting Carleman weight for $-h^2\Delta$. Set
$$
\mathcal L_{\varphi}:=e^{\varphi/h}(h^{2m}\mathcal L_{A,q})e^{-\varphi/h}.
$$
Then by Proposition~\ref{adjoints of D_A and m_q} we have
$$
\<\mathcal L_{\varphi} u,\overline{v}\>_{\Omega}= \<u,\overline{\mathcal L_{\varphi}^* v}\>_{\Omega},\quad u,v\in C^\infty_0(\Omega),
$$
where $\mathcal L_{\varphi}^*=e^{-\varphi/h}(h^2\mathcal L_{\overline{A},\overline{q}+D\cdot\overline{A}})e^{\varphi/h}$ is the formal adjoint of $\mathcal L_{\varphi}$, and $\langle \cdot,\cdot\rangle_{\Omega}$ is the distribution duality on $\Omega$. We have that
$$
\mathcal L_{\varphi}^*:C^\infty_0(\Omega)\to H^{-m}(\R^n)\cap\mathcal{E}'(\Omega)
$$
is bounded. Therefore, the estimate \eqref{Carleman est L_{A,q}} holds for $\mathcal L_{\varphi}^*$, since $-\varphi$ is a limiting Carleman weight as well.
\medskip

To construct the complex geometric optics solutions for the operator $\mathcal L_{A,q}$, we need to convert the Carleman estimate \eqref{Carleman est L_{A,q}} for $\mathcal L_{\varphi}^*$ into the following solvability result. The proof is essentially well-known, and we include it here for the convenience of the reader. In what follows, we shall write
$$
\|u\|_{H^m_{\rm scl}(\Omega)}=\sum_{|\alpha|\le m}\|(h\p)^\alpha u\|_{L^2(\Omega)},
$$
$$
\|v\|_{H^{-m}_{\rm scl}(\Omega)}=\sup_{0\neq\phi\in C^\infty_0(\Omega)}\frac{|\<v,\phi\>_{\Omega}|}{\|\phi\|_{H^m_{\rm scl}(\Omega)}}=\sup_{0\neq f\in H^m_0(\Omega)}\frac{|\<v,f\>_{\Omega}|}{\|f\|_{H^m_{\rm scl}(\Omega)}}.
$$
\begin{Proposition}\label{Solvability L_{A,q}}
Let $A\in W^{-\frac{m-2}{2},\frac{2n}{m}}(\R^n,\C^n)\cap \mathcal E'(\overline{\Omega},\C^n)$ and $q\in W^{-\frac{m}{2},\frac{2n}{m}}(\R^n,\C)\cap \mathcal E'(\overline{\Omega},\C)$, and let $\varphi$ be a limiting Carleman weight for the semiclassical Laplacian on $\tilde \Omega$. If $h>0$ is small enough, then for any $v\in H^{-m}(\Omega)$, there is a solution $u\in H^m(\Omega)$ of the equation
$$
e^{\varphi/h}(h^{2m}\mathcal L_{A,q})e^{-\varphi/h}u=v\quad\textrm{in}\quad \Omega,
$$
which satisfies
$$
\|u\|_{H^m_{\rm scl}(\Omega)}\lesim \frac{1}{h^m}\|v\|_{H^{-m}_{\rm scl}(\Omega)}.
$$
\end{Proposition}

\begin{proof}
Let $v\in H^{-m}(\Omega)$ and let us consider the following complex linear functional,
$$
L: \mathcal L_{\varphi}^* C_0^\infty(\Omega)\to \C, \quad \mathcal L_{\varphi,\alpha}^* w \mapsto \langle w, \overline{v}\rangle_\Omega. 
$$
By the Carleman estimate \eqref{Carleman est L_{A,q}} for $\mathcal L_{\varphi}^*$, the map $L$ is well-defined.  Let $w\in C_0^\infty(\Omega)$. Then we have
\begin{align*}
|L(\mathcal L_{\varphi}^* w)|=|\langle w, \overline{v}\rangle_\Omega|&\le \|w\|_{H^m_{\rm scl}(\R^n)}\|v\|_{H^{-m}_{\rm scl}(\Omega)}\\
&\le \frac{C}{h^m}\|v\|_{H^{-m}_{\rm scl}(\Omega)}\|\mathcal L_{\varphi}^* w\|_{H^{-m}_{\rm scl}(\R^n)}.
\end{align*}
By the Hahn-Banach theorem, we may extend $L$ to a linear continuous functional $\tilde L$ on $H^{-m}(\R^n)$, without increasing its norm. 
By the Riesz representation theorem, there exists $u\in H^m(\R^n)$ such that for all $\psi\in H^{-m}(\R^n)$,
$$
\tilde L(\psi)=\langle \psi,\overline{u}\rangle_{\R^n}, \quad \textrm{and}\quad \|u\|_{H^m_{\rm scl}(\R^n)}\le \frac{C}{h^m}\|v\|_{H^{-m}_{\rm scl}(\Omega)}. 
$$
Let us now show that $\mathcal L_{\varphi} u=v$ in $\Omega$. To that end, let $w\in C_0^\infty(\Omega)$. Then 
$$
\langle \mathcal L_{\varphi} u,\overline{w}\rangle_\Omega=\langle u,\overline{\mathcal L_{\varphi}^*w}\rangle_{\R^n} =\overline{\tilde L(\mathcal L_{\varphi}^*w)}=\overline{\langle w,\overline{v}\rangle_\Omega}=\langle v,\overline{w}\rangle_\Omega. 
$$
The proof is complete. 
\end{proof}

Our next goal is to construct the complex geometric optics solutions for the equation $\mathcal{L}_{A,q}u=0$ in $\Omega$ with $A\in W^{-\frac{m-2}{2},\frac{2n}{m}}(\R^n,\C^n)\cap \mathcal E'(\overline{\Omega},\C^n)$ and $q\in W^{-\frac{m}{2},\frac{2n}{m}}(\R^n,\C)\cap \mathcal E'(\overline{\Omega},\C)$ using the solvability result Proposition~\ref{Solvability L_{A,q}}. Complex geometric optics solutions are the solutions of the following form,
\begin{equation}\label{form of CGOs}
u(x,\zeta;h)=e^{\frac{ix\cdot \zeta}{h}} (a(x,\zeta)+h^{m/2} r(x,\zeta; h)),
\end{equation}
where $\zeta\in \C^n$ such that $\zeta\cdot \zeta=0$, $|\zeta|\sim 1$, $a\in C^\infty(\overline{\Omega})$ is an amplitude, $r$ is a correction term, and $h>0$ is a small parameter.

Let us conjugate $h^{2m}\mathcal L_{A,q}$ by $e^{ix\cdot\zeta/h}$. We have
\begin{multline}\label{expression for conjugated H_{A,q}}
e^{\frac{-ix\cdot \zeta}{h}} h^{2m}\mathcal L_{A,q} e^{\frac{ix\cdot \zeta}{h}}\\
=(-h^2\Delta-2i\zeta\cdot h\nabla)^m+h^{2m}D_A+h^{2m-1}m_{A\cdot \zeta}+h^{2m}m_q.
\end{multline}
We shall consider $\zeta$ depending slightly on $h$, i.e. $\zeta=\zeta_0+\zeta_1$ with $\zeta_0$ being independent of $h$ and $\zeta_1=\mathcal{O}(h)$ as $h\to 0$. We also assume that $|\Re \zeta_0|=|\Im \zeta_0|=1$. Then we can write \eqref{expression for conjugated H_{A,q}} as follows
\begin{align*}
e^{\frac{-ix\cdot \zeta}{h}}h^{2m}\mathcal L_{A,q} e^{\frac{ix\cdot \zeta}{h}}&=(-h^2\Delta-2i\zeta_0\cdot h\nabla-2i\zeta_1\cdot h\nabla)^m+h^{2m}D_A\\
&\qquad+h^{2m-1}m_{A\cdot(\zeta_0+\zeta_1)}+h^{2m}m_q.
\end{align*}
Then \eqref{form of CGOs} is a solution to $\mathcal L_{A,q}u=0$ if and only if 
\begin{multline}\label{equation for studying order of h}
e^{\frac{-ix\cdot\zeta}{h}} h^{2m} \mathcal{L}_{A,q}(e^{\frac{ix\cdot\zeta}{h}}h^{m/2}r)=-e^{\frac{-ix\cdot \zeta}{h}} h^{2m}\mathcal{L}_{A,q} (e^{\frac{ix\cdot \zeta}{h}} a)\\
=-\sum_{k=0}^{m}\frac{m!}{k!(m-k)!}(-h^2\Delta-2i\zeta_1\cdot h\nabla)^{m-k}(-2i\zeta_0\cdot h\nabla)^{k} a\\
-h^{2m}D_Aa-h^{2m-1}m_{A\cdot(\zeta_0+\zeta_1)}(a)-h^{2m}m_q(a)\quad\textrm{in}\quad\Omega.
\end{multline}
If $a\in C^\infty(\overline{\Omega})$ satisfies
$$
(\zeta_0\cdot\nabla)^{k_0}a=0\quad\textrm{in}\quad\Omega
$$
for some $k_0\ge 1$ integer, then, using the fact that $\zeta_1=\mathcal O(h)$, one can show that the lowest order of $h$ on the right-hand side of \eqref{equation for studying order of h} is $k_0-1+2(m-k_0+1)=2m-k_0+1$. In order to get
$$
\|e^{\frac{-ix\cdot \zeta}{h}} h^{2m}\mathcal{L}_{A,q} (e^{\frac{ix\cdot \zeta}{h}} a)\|_{H^{-m}_{\rm scl}(\Omega)}\le\mathcal{O}(h^{m+m/2}),
$$
we should choose $k_0$ satisfying $2m-k_0+1\ge m+m/2$ and hence $k_0\le (m+2)/2$. Since $m\ge 2$, we should choose $a\in C^\infty(\overline{\Omega})$, satisfying the following transport equation, 
\begin{equation}\label{eq_transport}
(\zeta_0\cdot \nabla)^{2} a=0\quad \textrm{in}\quad \Omega.
\end{equation}
The choice of such $a$ is clearly possible. Having chosen the amplitude $a$ in this way, we obtain the following equation for $r$,
\begin{multline}\label{Hhr=O(h^m+1)}
e^{\frac{-ix\cdot\zeta}{h}} h^{2m} \mathcal{L}_{A,q}(e^{\frac{ix\cdot\zeta}{h}}h^{m/2}r)=-e^{\frac{-ix\cdot \zeta}{h}} h^{2m}\mathcal{L}_{A,q} (e^{\frac{ix\cdot \zeta}{h}} a)\\
=-(-h^2\Delta-2i\zeta_1\cdot h\nabla)^{m}a-m(-h^2\Delta-2i\zeta_1\cdot h\nabla)^{m-1}(-2i\zeta_0\cdot h\nabla) a\\
-h^{2m}D_Aa-h^{2m-1}m_{A\cdot(\zeta_0+\zeta_1)}(a)-h^{2m}m_q(a)=:g\text{ in }\Omega.
\end{multline}
Notice that $g$ belongs to $H^{-m}(\Omega)$ and we would like to estimate $\|g\|_{H^{-m}_{\rm scl}(\Omega)}$. To that end, we let $\psi\in C^\infty_0(\Omega)$ such that $\psi\neq 0$. Then using the fact that $\zeta_1=\mathcal O(h)$, we get by the Cauchy-Schwarz inequality,
\begin{multline}\label{highest order part}
|\((-h^2\Delta-2i\zeta_1\cdot h\nabla)^{m}a,\psi\)_{L^2(\Omega)}|\\
+|\((-h^2\Delta-2i\zeta_1\cdot h\nabla)^{m-1}(-2i\zeta_0\cdot h\nabla)a,\psi\)_{L^2(\Omega)}|\\
\le\mathcal O(h^{2m-1})\|\psi\|_{L^2(\Omega)}\le\mathcal O(h^{2m-1})\|\psi\|_{H^m_{\rm scl}(\Omega)}.
\end{multline}
In the case $m\ge 3$, we use Proposition~\ref{product of Sobolev spaces} with $p=q=2$, $s_1=\frac{m-2}{2}$ and $s_2=\frac{m}{2}$ (assuming $m<n$), to get
\begin{align*}
|\<h^{2m-1}m_{A\cdot(\zeta_0+\zeta_1)}(a),\psi\>_\Omega|&\le Ch^{2m-1}\|A\|_{W^{-\frac{m-2}{2},\frac{2n}{m}}(\R^n)}\|a\psi\|_{W^{\frac{m-2}{2},\frac{2n}{2n-m}}(\R^n)}\\
&\le \mathcal O(h^{2m-1})\|\psi\|_{H^{\frac{m-2}{2}}(\Omega)}\le \mathcal O(h^{\frac{3m}{2}})\|\psi\|_{H^\frac{m}{2}_{\rm scl}(\Omega)}\\
&\le \mathcal O(h^{\frac{3m}{2}})\|\psi\|_{H^m_{\rm scl}(\Omega)}.
\end{align*}
When $m=2$, we use H\"older's inequality, and obtain
\begin{align*}
|\<h^{3}m_{A\cdot(\zeta_0+\zeta_1)}(a),\psi\>_\Omega|&\le Ch^{3}\|A\|_{L^{\frac{2n}{2}}(\R^n)}\|a\psi\|_{L^{\frac{2n}{2n-2}}(\R^n)}\\
&\le \mathcal O(h^3)\|a\|_{L^{\frac{2n}{n-2}}(\R^n)}\|\psi\|_{L^{\frac{2n}{n}}(\R^n)}\\
&\le \mathcal O(h^{3})\|\psi\|_{L^2(\Omega)}\le \mathcal O(h^{3})\|\psi\|_{H^2_{\rm scl}(\Omega)}.
\end{align*}
Therefore, for $m\ge 2$, we get
\begin{equation}\label{Aeta part}
|\<h^{2m-1}m_{A\cdot(\zeta_0+\zeta_1)}(a),\psi\>_\Omega|\le \mathcal O(h^{\frac{3m}{2}})\|\psi\|_{H^m_{\rm scl}(\Omega)}.
\end{equation}

Similarly, in the case $m\ge 3$, we use Proposition~\ref{product of Sobolev spaces} with $p=q=2$, $s_1=\frac{m-2}{2}$ and $s_2=\frac{m}{2}$ (assuming $m<n$), to get
\begin{align*}
|\<h^{2m}D_A(a),\psi\>_\Omega|&\le Ch^{2m}\|A\|_{W^{-\frac{m-2}{2},\frac{2n}{m}}(\R^n)}\|\psi Da\|_{W^{\frac{m-2}{2},\frac{2n}{2n-m}}(\R^n)}\\
&\le \mathcal O(h^{2m})\|\psi\|_{H^{\frac{m}{2}}(\Omega)}\le\mathcal O(h^{\frac{3m}{2}})\|\psi\|_{H^{\frac{m}{2}}_{\rm scl}(\Omega)}\\
&\le\mathcal O(h^{\frac{3m}{2}})\|\psi\|_{H^m_{\rm scl}(\Omega)},
\end{align*}
When $m=2$, we again use H\"older's inequality and Sobolev embeddings $H^{1}(\R^n)\subset L^{\frac{2n}{n-2}}(\R^n)$ (see \cite[Chapter~13, Proposition~6.4]{Taylor3}), and obtain
\begin{align*}
|\<h^{4}D_A(a),\psi\>_\Omega|&\le h^{4}\|A\|_{L^{\frac{2n}{2}}(\R^n)}\|\psi Da\|_{L^{\frac{2n}{2n-2}}(\R^n)}\\
&\le \mathcal O(h^4)\|Da\|_{L^{\frac{2n}{n}}(\R^n)}\|\psi\|_{L^{\frac{2n}{n-2}}(\R^n)}\\
&\le \mathcal O(h^{4})\|\psi\|_{H^{1}(\Omega)}\le \mathcal O(h^{3})\|\psi\|_{H^1_{\rm scl}(\Omega)}.
\end{align*}
Therefore, for $m\ge 2$, we get
\begin{equation}\label{ADa part}
|\<h^{2m-1} D_A(a),\psi\>_\Omega|\le \mathcal O(h^{\frac{3m}{2}})\|\psi\|_{H^m_{\rm scl}(\Omega)}.
\end{equation}

Finally, using Proposition~\ref{product of Sobolev spaces}, we show that
\begin{align*}
|\<h^{2m}m_q(a),\psi\>_\Omega|&\le h^{2m}\|q\|_{W^{-\frac m2,\frac{2n}{m}}(\R^n)}\|a\psi\|_{W^{\frac m2,\frac{2n}{2n-m}}(\R^n)}\\
&\le \mathcal O(h^{2m})\|\psi\|_{H^{\frac m2}(\R^n)}\le \mathcal O(h^{\frac{3m}{2}})\|\psi\|_{H^m_{\rm scl}(\Omega)}.
\end{align*}
Thus, combining this together with the estimates \eqref{highest order part} , \eqref{Aeta part}, \eqref{ADa part} in \eqref{Hhr=O(h^m+1)}, and using \eqref{H^-m characterization} and $m\ge 2$, we can conclude that
$$
\|g\|_{H^{-m}_{\rm scl}(\Omega)}\le\mathcal O(h^{\frac{3m}{2}})\le\mathcal O(h^{m+m/2}).
$$
Thanks to this and Proposition~\ref{Solvability L_{A,q}}, for $h>0$ small enough, there exists a solution $r\in H^m(\Omega)$ of \eqref{Hhr=O(h^m+1)} such that
$$
\|h^{m/2} r\|_{H^m_{\rm scl}(\Omega)}\lesim \frac{1}{h^m}\|e^{\frac{-ix\cdot \zeta}{h}} h^{2m}\mathcal{L}_{A,q} (e^{\frac{ix\cdot \zeta}{h}} a)\|_{H^{-m}_{\rm scl}(\Omega)}=\frac{1}{h^m}\|g\|_{H^{-m}_{\rm scl}(\Omega)}\lesim h^{m/2}.
$$
Therefore, $\|r\|_{H^m_{\rm scl}(\Omega)}=\mathcal O(1)$. The discussion of this section can be summarized in the following proposition.
\begin{Proposition}\label{CGO for L}
Let $\Omega\subset \R^n$, $n\ge 3$, be a bounded open set with $C^\infty$ boundary, and let $m\ge 2$ be an integer such that $n>m$. Suppose that $A \in W^{-\frac{m-2}{2},\frac{2n}{m}}(\R^n,\C^n)\cap\mathcal E'(\overline \Omega)$ and $q \in W^{-\frac{m}{2},\frac{2n}{m}}(\R^n,\C)\cap\mathcal E'(\overline \Omega)$, and let $\zeta\in \C^n$ be such that $\zeta\cdot\zeta=0$, $\zeta=\zeta_0+\zeta_1$ with $\zeta_0$ being independent of $h>0$, $|\Re\zeta_0|=|\Im \zeta_0|=1$, and $\zeta_1=\mathcal{O}(h)$ as $h\to 0$. Then for all $h>0$ small enough, there exists a solution $u(x,\zeta;h)\in H^m(\Omega)$ to the equation $\mathcal L_{A,q}u=0$ in $\Omega$, of the form
$$
u(x,\zeta;h)=e^{ix\cdot\zeta/h}(a(x,\zeta_0)+h^{m/2}r(x,\zeta;h)),
$$
where the function  $a(\cdot,\zeta_0)\in C^\infty(\overline\Omega)$ satisfies \eqref{eq_transport} and the remainder term $r$ is such that $\|r\|_{H^m_{\emph{\rm scl}}(\Omega)}=\mathcal O(1)$ as $h\to 0$.
\end{Proposition}

\section{Proof of Theorem~\ref{main th}}\label{Proof}
The first ingredient in the proof of Theorem~\ref{main th} is a standard reduction to a larger domain; see \cite{SyU}. For the proof we follow \cite[Proposition~3.2]{KU} and \cite[Lemma~4.2]{S}.

\begin{Proposition}\label{reduction to larger domain}
Let $\Omega,\Omega'\subset\R^n$ be two bounded opens sets such that $\Omega\subset\subset \Omega'$ and $\p\Omega$ being $C^\infty$. Let $A_1,A_2\in W^{-\frac{m-2}{2},\frac{2n}{m}}(\R^n,\C^n)\cap\mathcal E'(\overline \Omega)$ and $q_1,q_2\in W^{-\frac{m}{2},\frac{2n}{m}}(\R^n,\C)\cap\mathcal E'(\overline \Omega)$. If $\mathcal N_{A_1,q_1}=\mathcal N_{A_2,q_2}$, then $\mathcal N_{A_1,q_1}'=\mathcal N_{A_2,q_2}'$, where $\mathcal N_{A_j,q_j}'$ denotes the set of the Dirichlet-to-Neumann map for $\mathcal L_{A_j,q_j}$ in $\Omega'$, $j=1,2$.
\end{Proposition}
\begin{proof}
Let $f'\in \prod_{j=0}^m H^{m-j-1/2}(\p\Omega')$ and let $u_1'\in H^{m}(\Omega')$ be a unique solution to $\mathcal L_{A_1,q_1}u_1'=0$ in $\Omega'$ with $\gamma u'_1=f'$ on $\p\Omega'$, where $\gamma'$ denotes the Dirichlet trace on $\p\Omega'$. Let $u_1=u_1'|_{\Omega}\in H^{m}(\Omega)$ and $f=\gamma u_1$. Since $\mathcal N_{A_1,q_1}=\mathcal N_{A_2,q_2}$, we can guarantee the existence of $u_2\in H^{m}(\Omega)$ satisfying $\mathcal L_{A_2,q_2}u_2=0$ and $\gamma u_2=f$. In particular $\varphi:=u_2-u_1\in H^{m}_0(\Omega)\subset H^{m}_0(\Omega')$. We define
\begin{equation}\label{definition of u'_2}
u_2'=u_1'+\varphi\in H^{m}(\Omega'),
\end{equation}
so that we have $u'_2=u_2$ in $\Omega$. It follows that $\gamma' u_2'=\gamma' u_1'=f'$ on $\p\Omega'$.
\medskip

Next, we show that $u_2'$ satisfies $\mathcal L_{A_2,q_2}u'_2=0$ in $\Omega'$. For this, let $\psi\in C^\infty_0(\Omega')$. Then we have
$$
\<\mathcal L_{A_2,q_2}u'_2,\psi\>_{\Omega'}=\sum_{|\alpha|=m}\frac{m!}{\alpha!}(D^\alpha u_2',\overline{D^\alpha \psi})_{L^2(\Omega')}+\<D_{A_2}(u_2'),\psi\>_{\Omega'}+\<m_{q_2}(u_2'),\psi\>_{\Omega'}.
$$
Since $A_2=0$ and $q_2=0$ outside of $\overline{\Omega}$, by \eqref{definition of u'_2}, with $\varphi\in H^m_0(\Omega)$, we can rewrite the above equality as
\begin{align*}
\<\mathcal L_{A_2,q_2}u'_2,\psi\>_{\Omega'}&=\sum_{|\alpha|=m}\frac{m!}{\alpha!}(D^\alpha u_1',\overline{D^\alpha \psi})_{L^2(\Omega')}+\sum_{|\alpha|=m}\frac{m!}{\alpha!}(D^\alpha \varphi,\overline{D^\alpha(\psi|_\Omega)})_{L^2(\Omega)}\\
&\qquad+B_{A_2}(u_2',\psi|_{\Omega})+b_{q_2}(u_2',\psi|_{\Omega})\\
&=\sum_{|\alpha|=m}\frac{m!}{\alpha!}(D^\alpha u_1',\overline{D^\alpha \psi})_{L^2(\Omega')}-\sum_{|\alpha|=m}\frac{m!}{\alpha!}(D^\alpha u_1,\overline{D^\alpha(\psi|_\Omega)})_{L^2(\Omega)}\\
&\qquad+\sum_{|\alpha|=m}\frac{m!}{\alpha!}(D^\alpha u_2,\overline{D^\alpha(\psi|_\Omega)})_{L^2(\Omega)}+B_{A_2}(u_2,\psi|_{\Omega})\\
&\qquad+b_{q_2}(u_2,\psi|_{\Omega}).
\end{align*}
Note that
\begin{align*}
\<\mathcal N_{A_2,q_2}f,\gamma(\psi|_{\Omega})\>_{\p\Omega}&=\sum_{|\alpha|=m}\frac{m!}{\alpha!}(D^\alpha u_2,\overline{D^\alpha(\psi|_\Omega)})_{L^2(\Omega)}\\
&\qquad+B_{A_2}(u_2,\psi|_{\Omega})+b_{q_2}(u_2,\psi|_{\Omega}).
\end{align*}
Therefore, we have
\begin{align*}
\<\mathcal L_{A_2,q_2}u'_2,\psi\>_{\Omega'}&=\sum_{|\alpha|=m}\frac{m!}{\alpha!}(D^\alpha u_1',\overline{D^\alpha \psi})_{L^2(\Omega')}-\sum_{|\alpha|=m}\frac{m!}{\alpha!}(D^\alpha u_1,\overline{D^\alpha(\psi|_\Omega)})_{L^2(\Omega)}\\
&\qquad+\<\mathcal N_{A_2,q_2}f,\psi|_{\p\Omega}\>_{\p\Omega}.
\end{align*}
Since $\mathcal N_{A_1,q_1}=\mathcal N_{A_2,q_2}$ and since
\begin{align*}
\<\mathcal N_{A_1,q_1}f,\gamma(\psi|_{\Omega})\>_{\p\Omega}&=\sum_{|\alpha|=m}\frac{m!}{\alpha!}(D^\alpha u_1,\overline{D^\alpha(\psi|_\Omega)})_{L^2(\Omega)}\\
&\qquad+B_{A_1}(u_1,\psi|_{\Omega})+b_{q_1}(u_1,\psi|_{\Omega}),
\end{align*}
we come to
\begin{align*}
\<\mathcal L_{A_2,q_2}u'_2,\psi\>_{\Omega'}&=\sum_{|\alpha|=m}\frac{m!}{\alpha!}(D^\alpha u_1',\overline{D^\alpha \psi})_{L^2(\Omega')}+B_{A_1}(u_1,\psi|_{\Omega})+b_{q_1}(u_1,\psi|_{\Omega}).
\end{align*}
Using that $A_1=0$ and $q_1=0$ outside $\overline\Omega$, we obtain
\begin{align*}
\<\mathcal L_{A_2,q_2}u'_2,\psi\>_{\Omega'}&=\sum_{|\alpha|=m}\frac{m!}{\alpha!}(D^\alpha u_1',\overline{D^\alpha \psi})_{L^2(\Omega')}+\<D_{A_1}(u_1'),\psi\>_{\Omega'}+\<m_{q_1}(u_1'),\psi\>_{\Omega'}\\
&=\<\mathcal L_{A_1,q_1}u'_1,\psi\>_{\Omega'}=0.
\end{align*}
This shows that $\mathcal L_{A_2,q_2}u'_2=0$ in $\Omega'$.
\medskip

Using the analogous arguments one can show that $\mathcal N_{A_2,q_2}'f'=\mathcal N_{A_1,q_1}'f'$ on $\p\Omega'$, which finishes the proof.
\end{proof}

The second ingredient is the derivation of the following integral identity based on the assumption that $\mathcal N_{A_1,q_1}=\mathcal N_{A_2,q_2}$.

\begin{Proposition}\label{integral identity prop}
Let $\Omega\subset\R^n$, $n\ge 3$, be a bounded open set with $C^\infty$ boundary. Assume that $A_1,A_2\in W^{-\frac{m-2}{2},\frac{2n}{m}}(\R^n,\C^n)\cap\mathcal E'(\overline \Omega)$ and $q_1,q_2\in W^{-\frac{m}{2},\frac{2n}{m}}(\R^n,\C)\cap\mathcal E'(\overline \Omega)$. If $\mathcal N_{A_1,q_1}=\mathcal N_{A_2,q_2}$, then the following integral identity holds
\begin{equation}\label{integral identity}
\begin{aligned}
0&=B_{A_1-A_2}(u_1,\overline u_2)+b_{q_1-q_2}(u_1,\overline{u_2})\\
&=(D_{A_1-A_2}(u_1),\overline u_2)_\Omega+(m_{q_1-q_2}(u_1),\overline u_2)_\Omega
\end{aligned}
\end{equation}
for any $u_1,u_2\in H^m(\Omega)$ satisfying $\mathcal L_{A_1,q_1}u_1=0$ in $\Omega$ and $\mathcal L^*_{{A}_2,{q}_2}u_2=0$ in $\Omega$, respectively.
\end{Proposition}
Recall that $\mathcal L_{A,q}^*=\mathcal L_{\overline{A},\overline{q}+D\cdot \overline A}$ is the formal adjoint of $\mathcal L_{A,q}$.
\begin{proof}
Since $u_2\in H^m(\Omega)$ satisfies $\mathcal L_{-A_2,q_2-D\cdot A_2}\overline{u}_2=0$, the following
\begin{equation}\label{Lu=0 in weak sense}
\begin{aligned}
0=\<\mathcal L_{-A_2,q_2-D\cdot A_2}\overline{u}_2,\psi\>_\Omega&=\sum_{|\alpha|=m}\frac{m!}{\alpha!}(D^\alpha\overline u_2,D^\alpha \psi)_{L^2(\Omega)}-\<D_{A_2}(\overline u_2),\psi\>_\Omega\\
&\qquad+\<m_{q_2-D\cdot A_2}(\overline u_2),\psi\>_\Omega,
\end{aligned}
\end{equation}
holds for every $\psi\in C^\infty_0(\Omega)$. Density and continuity imply that \eqref{Lu=0 in weak sense} holds also for all $\psi\in H^m_0(\Omega)$.
\medskip

The hypothesis that $\mathcal N_{A_1,q_1}=\mathcal N_{A_2,q_2}$ implies the existence of $v_2\in H^m(\Omega)$ satisfying $\mathcal L_{{A_2},{q_2}}v_2=0$ in $\Omega$ such
$$
\gamma u_1=\gamma v_2\quad\text{and}\quad \mathcal N_{A_1,q_1}\gamma u_1=\mathcal N_{A_2,q_2}\gamma v_2.
$$
Then $\psi=u_1-v_2\in H^m_0(\Omega)$. Hence, applying \eqref{Lu=0 in weak sense} we obtain
\begin{equation}\label{pre int identity}
\begin{aligned}
0&=\sum_{|\alpha|=m}\frac{m!}{\alpha!}(D^\alpha\overline u_2,D^\alpha (u_1-v_2))_{L^2(\Omega)}-\<D_{A_2}(\overline u_2),(u_1-v_2)\>_\Omega\\
&\qquad+\<m_{q_2-D\cdot A_2}(\overline u_2),(u_1-v_2)\>_\Omega.
\end{aligned}
\end{equation}
The equality $\<\mathcal N_{A_1,q_1}\gamma_1,\gamma\overline u_2\>_{\p\Omega}=\<\mathcal N_{A_2,q_2}\gamma v_2,\gamma\overline u_2\>_{\p\Omega}$ together with the definition \eqref{N_A,q} gives
\begin{align*}
0&=\sum_{|\alpha|=m}\frac{m!}{\alpha!}(D^\alpha (u_1-v_2),D^\alpha\overline u_2)_{L^2(\Omega)}-\(B_{A_1}(u_1,\overline u_2)-B_{A_2}(v_2,\overline u_2)\)\\
&\qquad+b_{q_1}(u_1,\overline u_2)-b_{q_2}(v_2,\overline u_2).
\end{align*}
Combining this with \eqref{pre int identity} and using Proposition~\ref{adjoints of D_A and m_q}, we derive the integral identity \eqref{integral identity} as desired.
\end{proof}

Let $A_1,A_2\in W^{-\frac{m-2}{2},\frac{2n}{m}}(\R^n,\C^n)\cap\mathcal E'(\overline \Omega)$ and $q_1,q_2\in W^{-\frac{m}{2}+\delta,\frac{2n}{m}}(\R^n,\C)\cap\mathcal E'(\overline \Omega)$, $0<\delta<1/2$, as in  the statement of Theorem~\ref{main th}. In order to show that $A_1=A_2$, we will need to use the Poincar\'e lemma for currents \cite{deRham}. We need this reduction to apply the Poincar\'e lemma for currents , which requires the domain to be simply connected. Therefore, we reduce the problem to a larger simply connected domain. In particular, to a ball.
\medskip

Let $B$ be an open ball in $\R^n$ such that $\Omega\subset\subset B$. According to Proposition~\ref{reduction to larger domain}, we know that $\mathcal N_{A_1,q_1}^B=\mathcal N_{A_2,q_2}^B$, where $\mathcal N_{A_j,q_j}^B$ denotes the Dirichlet-to-Neumann map for $\mathcal L_{A_j,q_j}$ in $B$, $j=1,2$. Now, by Proposition~\ref{integral identity prop}, the following integral identity holds
\begin{equation}\label{integral identity'}
B_{A_1-A_2}^B(u_1,\overline u_2)+b_{q_1-q_2}^B(u_1,\overline u_2)=0
\end{equation}
for any $u_1,u_2\in H^m(B)$ satisfying $\mathcal L_{A_1,q_1}u_1=0$ in $B$ and $\mathcal L^*_{{A}_2,{q}_2}u_2=0$ in $B$, respectively. Here and in what follows, by $B_{A_1-A_2}^B$ and $b_{q_1-q_2}^B$ we denote the bilinear forms corresponding to $A_1-A_2$ and $q_1-q_2$, respectively, defined (by means of \eqref{def of B_A b_q on Omega}) in the ball $B$.
\medskip

The main idea of the proof of Theorem~\ref{main th} is to use the integral identity \eqref{integral identity'} with $u_1,u_2\in H^m(B)$ being complex geometric optics solutions for the equations $\mathcal L_{A_1,q_1}u_1=0$ in $B$ and $\mathcal L^*_{{A}_2,{q}_2}u_2=0$ in $B$, respectively. In order to construct these solutions, consider $\xi,\mu_1,\mu_2\in\R^n$ such that $|\mu_1|=|\mu_2|=1$ and $\mu_1\cdot\mu_2=\xi\cdot\mu_1=\xi\cdot\mu_2=0$. For $h>0$, we set
$$
\zeta_1=\frac{h\xi}{2}+\sqrt{1-h^2\frac{|\xi|^2}{4}}\mu_1+i\mu_2,\quad \zeta_2=-\frac{h\xi}{2}+\sqrt{1-h^2\frac{|\xi|^2}{4}}\mu_1-i\mu_2.
$$
So we have $\zeta_j\cdot\zeta_j=0$, $j=1,2$, and $\zeta_1-\overline{\zeta}_2=h\xi$.
\medskip

By Proposition~\ref{CGO for L}, for all $h > 0$ small enough, there are solutions $u_1(\cdot, \zeta_1; h)$ and $u_2(\cdot, \zeta_2; h)$ in $H^m(B)$ to the equations
$$
\mathcal L_{A_1,q_1}u_1=0\text{ in }B\text{ and }\mathcal L^*_{{A}_2,{q}_2}u_2=0\text{ in }B,
$$
respectively, of the form
\begin{equation}\label{solutions}
\begin{aligned}
u_1(x,\zeta_1;h)&=e^{ix\cdot\zeta_1/h}(a_1(x,\mu_1+i\mu_2)+h^{m/2}r_1(x,\zeta_1;h))\\
u_2(x,\zeta_2;h)&=e^{ix\cdot\zeta_2/h}(a_2(x,\mu_1-i\mu_2)+h^{m/2}r_2(x,\zeta_2;h)),
\end{aligned}
\end{equation}
where the amplitudes $a_1(x,\mu_1+i\mu_2),a_2(x,\mu_1-i\mu_2)\in C^\infty(\overline B)$ satisfy the transport equations
\begin{equation}\label{transport equations 1}
((\mu_1+i\mu_2)\cdot\nabla)^2 a_1(x,\mu_1+i\mu_2)=0,\quad \textrm{in}\quad B,
\end{equation}
and
\begin{equation}\label{transport equations 2}
((\mu_1-i\mu_2)\cdot\nabla)^2 a_2(x,\mu_1-i\mu_2)=0,\quad \textrm{in}\quad B,
\end{equation}
and the remainder terms $r_1(\cdot,\zeta_1;h)$ and $r_2(\cdot,\zeta_2;h)$ satisfy
\begin{equation}\label{r decay}
\|r_j\|_{H^m_{\rm scl}(B)}=\mathcal O(1),\quad j=1,2.
\end{equation}
We substitude $u_1$ and $u_2$ given by \eqref{solutions} into \eqref{integral identity}, and get
\begin{equation}\label{pre limiting equality}
\begin{aligned}
0&=\frac1h b^B_{\zeta_1\cdot(A_1-A_2)}\(a_1+h^{m/2}r_1,e^{ix\cdot\xi}(\overline a_2+h^{m/2}\overline r_2)\)\\
&\quad+B^B_{A_1-A_2}\(a_1+h^{m/2}r_1, e^{ix\cdot\xi}(\overline a_2+h^{m/2}\overline r_2)\)\\
&\quad+b^B_{q_1-q_2}\(a_1+h^{m/2}r_1,e^{ix\cdot\xi}(\overline a_2+h^{m/2}\overline r_2)\).
\end{aligned}
\end{equation}
Multiplying this by $h$ and letting $h\to+0$, we obtain that
\begin{equation}\label{limiting equality 1}
b^B_{(\mu_1+i\mu_2)\cdot(A_1-A_2)}\(a_1,e^{ix\cdot\xi}\overline a_2\)=0.
\end{equation}
Here we have used \eqref{r decay}, Proposition~\ref{m_q and D_A are bounded maps on Omega} and the fact that $a_1,a_2\in C^\infty(\overline B)$ to conclude that
\begin{align*}
|B^B_{A_1-A_2}&\(a_1+h^{m/2}r_1, e^{ix\cdot\xi}(\overline a_2+h^{m/2}\overline r_2)\)|\\
&\le C\|A_1-A_2\|_{W^{-\frac{m-2}{2},\frac{2n}{m}}(\R^n)} \|a_1+h^{m/2}r_1\|_{H^{\frac m2}(B)}\|\overline a_2+h^{m/2}\overline r_2\|_{H^{\frac m2}(B)}\\
&\le C(\|a_1\|_{H^{\frac m2}(B)}+\|r_1\|_{H^{\frac m2}_{\rm scl}(B)})(\|\overline a_2\|_{H^{\frac m2}(B)}+\|\overline r_2\|_{H^{\frac m2}_{\rm scl}(B)})\le \mathcal O(1).
\end{align*}
and
\begin{align*}
|b^B_{q_1-q_2}&\(a_1+h^{m/2}r_1,e^{ix\cdot\xi}(\overline a_2+h^{m/2}\overline r_2)\)|\\
&\le C\|q_1-q_2\|_{W^{-\frac{m}{2},\frac{2n}{m}}(\R^n)} \|a_1+h^{m/2}r_1\|_{H^{\frac m2}(B)}\|\overline a_2+h^{m/2}\overline r_2\|_{H^{\frac m2}(B)}\\
&\le C(\|a_1\|_{H^{\frac m2}(B)}+h^{m/2}\|r_1\|_{H^{\frac m2}_{\rm scl}(B)})(\|\overline a_2\|_{H^{\frac m2}(B)}+h^{m/2}\|\overline r_2\|_{H^{\frac m2}_{\rm scl}(B)})\\
&\le \mathcal O(1).
\end{align*}

Substituting $a_1=a_2=1$ in \eqref{limiting equality 1}, we obtain
\begin{equation}\label{limiting equality 2}
\left\<(\mu_1+i\mu_2)\cdot(A_1-A_2),e^{ix\cdot\xi}\right\>_B=0.
\end{equation}
This implies that
\begin{equation}\label{limiting equality 3}
(\mu_1+i\mu_2)\cdot(\widehat A_1(\xi)-\widehat A_2(\xi))=0,\quad \text{for all}\quad \mu,\xi\in\R^n,\quad \mu\cdot\xi=0,
\end{equation}
where $\widehat{A}_j$ stands for the Fourier transform of $A_j$, $j=1,2$. It follows from \eqref{limiting equality 3} that
\begin{equation}\label{d(A1-A2)=0}
\p_{x_j}(A_{1,k}-A_{2,k})-\p_{x_k}(A_{1,j}-A_{2,j})=0\quad \text{in}\quad \Omega,\quad 1\le j,k\le n,
\end{equation}
in the sense of distributions. Indeed, for each $\xi=(\xi_1,\dots,\xi_n)$ and for $j\neq k$, $1\le j,k\le n$, consider the vector $\mu=\mu(\xi,j,k)$ such that $\mu_j=-\xi_k$, $\mu_k=\xi_j$ and all other components are equal to zero. Therefore, $\mu$ satisfy $\mu\cdot\xi=0$. Hence, using \eqref{limiting equality 3} we obtain
$$
\xi_j\cdot(\widehat A_{1,k}(\xi)-\widehat A_{2,k}(\xi))-\xi_k\cdot(\widehat A_{1,j}(\xi)-\widehat A_{2,j}(\xi))=0,
$$
which proves \eqref{d(A1-A2)=0} in the sense of distributions.
\medskip

Our goal is to show that $A_1=A_2$. Considering $A_1-A_2$ as a $1$-current and using the Poincar\'e lemma for currents, we conclude that there is $\psi\in\mathcal D'(\R^n)$ such that $d \psi=A_1-A_2\in W^{-\frac{m-2}{2},\frac{2n}{m}}(\R^n,\C^n)\cap\mathcal E'(B,\C^n)$; see \cite{deRham}. Note that $\psi$ is a constant, say $c\in \C$, outside $\overline B$ since $A_1-A_2=0$ in $\R^n\setminus\overline B$ (and also near $\p B$). Considering $\psi-c$ instead of $\psi$, we may and shall assume that $\psi\in \mathcal E'(\overline B,\C)$.
\medskip

To show that $A_1=A_2$, consider \eqref{limiting equality 1} with $a_2(\cdot,\mu_1-i\mu_2)=1$ and $a_1(\cdot,\mu_1+i\mu_2)$ satisfying
\begin{equation}\label{pre transport}
((\mu_1+i\mu_2)\cdot\nabla)a_1(x,\mu_1+i\mu_2)=1\quad\text{in}\quad B.
\end{equation}
The latter choice is possible thanks to \eqref{transport equations 1}, \eqref{transport equations 2} and the assumption that $m\ge 2$. The equation \eqref{pre transport} is just an inhomogeneous $\overline \p$-equation and one can solve it by setting
$$
a_1(x,\mu_1+i\mu_2)=\frac{1}{2\pi}\int_{\R^2}\frac{\chi(x-y_1\mu_1-y_2\mu_2)}{y_1+iy_2}\,dy_1\,dy_2,
$$
where $\chi\in C^\infty_0(\R^n)$ is such that $\chi\equiv 1$ near $\overline B$; see \cite[Lemma~4.6]{S}.
\medskip

From \eqref{limiting equality 1}, we have
$$
b^B_{(\mu_1+i\mu_2)\cdot\nabla \psi}(a_1,e^{ix\cdot\xi})=0.
$$
Using the fact that $\mu_1\cdot\xi=\mu_2\cdot\xi=0$, we obtain
\begin{align*}
0&=-b^B_{(\mu_1+i\mu_2)\cdot\nabla \psi}(a_1,e^{ix\cdot\xi})=-\left\<(\mu_1+i\mu_2)\cdot\nabla \psi, e^{ix\cdot\xi}a_1\right\>_B\\
&=\left\<\psi,e^{ix\cdot\xi} (\mu_1+i\mu_2)\cdot\nabla a_1\right\>_B=\<\psi,e^{ix\cdot\xi}\>_B.
\end{align*}
This gives $\hat\psi=0$, and hence we have $\psi=0$ in $B$, which completes the proof of $A_1=A_2$.
\medskip

To show that $q_1=q_2$, we substitude $A_1=A_2$ and $a_1=a_2=1$ into the identity \eqref{pre limiting equality} and obtain
$$
b^B_{q_1-q_2}\(1+h^{m/2}r_1,(1+h^{m/2}\overline r_2)e^{ix\cdot\xi}\)=0.
$$
Letting $h\to 0^+$, we get $\widehat{q}_1(\xi)-\widehat{q}_2(\xi)=0$ for all $\xi\in\R^n$. Let us justify this last statement. We will only consider the term $b^{B}_{q_2 - q_1}(h^{m/2}r_1,e^{i x \cdot \xi})$. The justification for the other two terms follows similarly. Since $q_1,q_2\in W^{-\frac{m}{2}+\delta,\frac{2n}{m}}(\R^n,\C)\cap\mathcal E'(\overline \Omega)$ for some $0<\delta<1/2$, we use Proposition~\ref{product of Sobolev spaces} with $p=q=2$, $s_1=\frac{m}{2}-\delta$, $s_2=\frac{m}{2}$ and $r=2n/(2n-m)$ to get
\begin{align*}
|b^{B}_{q_2 - q_1}(h^{m/2}r_1,e^{i x \cdot \xi})| &\leq C h^{m/2}\|q_2 - q_1\|_{W^{-\frac{m}{2}+\delta,\frac{2n}{m}}(\mathbb{R}^n)}\|r_1e^{i x \cdot \xi}\|_{W^{\frac{m}{2}-\delta,\frac{2n}{2n-m}}(\R^n)}\\
&\leq Ch^{m/2} \|q_2 - q_1\|_{W^{-\frac{m}{2}+\delta,\frac{2n}{m}}(\mathbb{R}^n)} \|e^{i x \cdot \xi}\|_{H^{\frac{m}{2} }(B)} \|r_1\|_{H^{\frac{m}{2} - \delta}(B)}\\
&\leq Ch^{\delta}\|q_2 - q_1\|_{W^{-\frac{m}{2}+\delta,\frac{2n}{m}}(\mathbb{R}^n)} \|e^{i x \cdot \xi}\|_{H^{\frac{m}{2} }(B)} \|h^{\frac{m}{2}-\delta}r_1\|_{H^{\frac{m}{2} - \delta}(B)}\\
&\le\mathcal{O}(h^{\delta}) \|r_1\|_{H^{\frac{m}{2} - \delta}_{\rm scl}(B)}\le\mathcal{O}(h^{\delta}).
\end{align*}
This implies that $q_1=q_2$ in $B$ completing the proof of Theorem~\ref{main th}.

\appendix
\section{Mapping properties of $D_A$ and $m_q$}\label{appendix a}
Let $\Omega\subset \R^n$, $n\ge 3$, be a bounded open set with $C^\infty$ boundary, and $m\ge 2$ be an integer such that $n>m$. Let $A\in W^{-\frac{m-2}{2},\frac{2n}{m}}(\R^n,\C^n)\cap \mathcal E'(\overline{\Omega},\C^n)$ and $q\in W^{-\frac{m}{2},\frac{2n}{m}}(\R^n,\C)\cap \mathcal E'(\overline{\Omega},\C)$, where $W^{s,p}(\R^n)$ is the standard $L^p$-based Sobolev space on $\R^n$, $s\in\R$ and $1<p<\infty$. The reader is referred to \cite{Trieb} for properties of these spaces.
\medskip

We start with considering the bilinear forms $B_A^{\R^n}$ and $b_q^{\R^n}$ on $H^m(\R^n)$ which are defined by
$$
B_A^{\R^n}(u,v)=\<A,vDu\>_{\R^n},\quad b_q^{\R^n}(u,v)=\<q,uv\>_{\R^n},\quad u,v\in H^m(\R^n).
$$
The following result shows that the forms $B_A^{\R^n}$ and $b_q^{\R^n}$ are bounded on $H^m(\R^n)$. The proof is based on a property on multiplication of functions in Sobolev spaces given in Proposition~\ref{product of Sobolev spaces}.
\begin{Proposition}\label{m_q and D_A are bounded maps}
The bilinear forms $B_A^{\R^n}$ and $b_q^{\R^n}$ on $H^m(\R^n)$ are bounded and satisfy
$$
|b_q^{\R^n}(u,v)|\le C\|q\|_{W^{-\frac{m}{2},\frac{2n}{m}}(\R^n)}\|u\|_{H^{\frac m2}(\R^n)}\|v\|_{H^{\frac m2}(\R^n)}
$$
and
$$
|B_A^{\R^n}(u,v)|\le C\|A\|_{W^{-\frac{m-2}{2},\frac{2n}{m}}(\R^n)}\|u\|_{H^{\frac m2}(\R^n)}\|v\|_{H^{\frac m2}(\R^n)}
$$ 
for all $u,v\in H^{m}(\R^n)$.
\end{Proposition}
\begin{proof}
First, we give the proof for the form $b_q^{\R^n}$. Using the duality between $W^{-\frac{m}{2},\frac{2n}{m}}(\R^n)$ and $W^{\frac{m}{2},\frac{2n}{2n-m}}(\R^n)$, and using Proposition~\ref{product of Sobolev spaces}, we conclude that for all $u,v\in H^m(\R^n)$
\begin{align*}
|b_q^{\R^n}(u,v)|=|\<q,uv\>_{\R^n}|&\le\|q\|_{W^{-\frac{m}{2},\frac{2n}{m}}(\R^n)}\|uv\|_{W^{\frac{m}{2},\frac{2n}{2n-m}}(\R^n)}\\
&\le C\|q\|_{W^{-\frac{m}{2},\frac{2n}{m}}(\R^n)}\|u\|_{H^{\frac{m}{2}}(\R^n)}\|v\|_{H^{\frac{m}{2}}(\R^n)}.
\end{align*}
Now, we give the proof for the form $B_A^{\R^n}$. In the case $m\ge 3$, using the duality between $W^{-\frac{m-2}{2},\frac{2n}{m}}(\R^n)$ and $W^{\frac{m-2}{2},\frac{2n}{2n-m}}(\R^n)$, and using Proposition~\ref{product of Sobolev spaces}  with $p=q=2$, $s_1=\frac{m-2}{2}$ and $s_2=\frac{m}{2}$ (assuming $m<n$), we conclude that
\begin{align*}
|B_A^{\R^n}(u,v)|=|\<A,v Du\>_{\R^n}|&\le\|A\|_{W^{-\frac{m-2}{2},\frac{2n}{m}}(\R^n)}\|v Du\|_{W^{\frac{m-2}{2},\frac{2n}{2n-m}}(\R^n)}\\
&\le C\|A\|_{W^{-\frac{m-2}{2},\frac{2n}{m}}(\R^n)}\|Du\|_{H^{\frac{m-2}{2}}(\R^n)}\|v\|_{H^{\frac{m}{2}}(\R^n)}\\
&\le C\|A\|_{W^{-\frac{m-2}{2},\frac{2n}{m}}(\R^n)}\|u\|_{H^{\frac m2}(\R^n)}\|v\|_{H^{\frac m2}(\R^n)}.
\end{align*}
When $m=2$, we use H\"older's inequality and Sobolev embedding $H^{1}(\R^n)\subset L^{\frac{2n}{n-2}}(\R^n)$ (see \cite[Chapter~13, Proposition~6.4]{Taylor3}), and obtain
\begin{align*}
|B_A^{\R^n}(u,v)|=|\<A,v Du\>_{\R^n}|&\le\|A\|_{L^{n}(\R^n)}\|v Du\|_{L^{\frac{2n}{2n-2}}(\R^n)}\\
&\le \|A\|_{L^{n}(\R^n)}\|Du\|_{L^{\frac{2n}{n}}(\R^n)}\|v\|_{L^{\frac{2n}{n-2}}(\R^n)}\\
&\le C\|A\|_{L^{n}(\R^n)}\|Du\|_{L^{2}(\R^n)}\|v\|_{H^{1}(\R^n)}\\
&\le C\|A\|_{L^{n}(\R^n)}\|u\|_{H^{1}(\R^n)}\|v\|_{H^{1}(\R^n)}.
\end{align*}
Therefore,
$$
|B_A^{\R^n}(u,v)|\le C\|A\|_{W^{-\frac{m-2}{2},\frac{2n}{m}}(\R^n)}\|u\|_{H^{\frac m2}(\R^n)}\|\psi\|_{H^{\frac m2}(\R^n)}
$$
for all $u,v\in H^m(\R^n)$.
\end{proof}

The bilinear forms $B_A$ and $b_q$ on $H^m(\Omega)$, which were defined in the introduction, can be rewritten as
\begin{equation}\label{def of B_A b_q on Omega}
B_A(u,v)=B_A^{\R^n}(\tilde u,\tilde v),\quad b_q(u,v)=b_q^{\R^n}(\tilde u,\tilde v),\quad u,v\in H^m(\Omega),
\end{equation}
where $\tilde u,\tilde v \in H^m(\R^n)$ are extensions of $u$ and $v$, respectively. First, we show that these definitions are well-defined, i.e. independent of the choice of extensions $\tilde u,\tilde v$. Indeed, let $u_1,u_2\in H^{m}(\R^n)$ be such that $u_1=u_2=u$ in $\Omega$, and let $v_1,v_2\in H^{m}(\R^n)$ be such that $v_1=v_2=v$ in $\Omega$. Then we need to show that
\begin{equation}\label{WTS for well defined}
B_A^{\R^n}(u_1-u_2,,v_1-v_2)=0\quad\text{and}\quad b_q^{\R^n}(u_1-u_2,v_1-v_2)=0.
\end{equation}
Since $A$ and $q$ are supported in $\overline{\Omega}$, for any $\phi,\psi\in \mathcal S(\R^n)$ with $\phi=\psi=0$ in $\Omega$, we have
$$
B_A^{\R^n}(\phi,\psi)=\<A,\psi D\phi\>_{\R^n}=0\quad\text{and}\quad b_q^{\R^n}(\phi,\psi)=\<q,\psi \phi\>_{\R^n}=0.
$$
Since $\mathcal S(\R^n)$ is dense in $H^{m}(\R^n)$ and $B_A^{\R^n}$ and $b_q^{\R^n}$ are continuous bilinear forms (by Proposition~\ref{m_q and D_A are bounded maps}), we get \eqref{WTS for well defined}.
\medskip

The next result shows that the bilinear forms $B_A$ and $b_q$ are bounded on $H^m(\Omega)$. This is a consequence of Proposition~\ref{m_q and D_A are bounded maps}.

\begin{Proposition}\label{m_q and D_A are bounded maps on Omega}
The bilinear forms $B_A$ and $b_q$ on $H^m(\Omega)$ are bounded and satisfy
$$
|b_q(u,v)|\le C\|q\|_{W^{-\frac{m}{2},\frac{2n}{m}}(\R^n)}\|u\|_{H^{\frac m2}(\Omega)}\|v\|_{H^{\frac m2}(\Omega)}
$$
and
$$
|B_A(u,v)|\le C\|A\|_{W^{-\frac{m-2}{2},\frac{2n}{m}}(\R^n)}\|u\|_{H^{\frac m2}(\Omega)}\|v\|_{H^{\frac m2}(\Omega)}
$$
for all $u,v\in H^{m}(\Omega)$.
\end{Proposition}

\begin{proof}
Let $u,v\in H^{m}(\Omega)$ and let $\widetilde \Omega$ be a bounded open neigbhborhood of $\overline{\Omega}$. Then there is a bounded linear map $E:H^{m}(\Omega)\to H^{m}_0(\widetilde \Omega)$ such that $E|_\Omega=\id$; see \cite[Theorem~6.44]{Fol}. Then according to estimates proven in Proposition~\ref{m_q and D_A are bounded maps}, we obtain
\begin{align*}
|b_q(u,v)|&=|b_q^{\R^n}(E(u)),E(v))|\\
&\le C\|q\|_{W^{-\frac{m}{2},2n}(\R^n)}\|E(u)\|_{H^{\frac m2}(\R^n)}\|E(v)\|_{H^{\frac m2}(\R^n)}\\
&\le C\|q\|_{W^{-\frac{m}{2},2n}(\R^n)}\|u\|_{H^{\frac m2}(\Omega)}\|v\|_{H^{\frac m2}(\Omega)}
\end{align*}
and
\begin{align*}
|B_A(u,v)|&=|B_A^{\R^n}(E(u),E(v))|\\
&\le C\|A\|_{W^{-\frac{m-2}{2},\frac{2n}{m}}(\R^n)}\|E(u)\|_{H^{\frac m2}(\R^n)}\|E(v)\|_{H^{\frac m2}(\R^n)}\\
&\le C\|A\|_{W^{-\frac{m-2}{2},\frac{2n}{m}}(\R^n)}\|u\|_{H^{\frac m2}(\Omega)}\|v\|_{H^{\frac m2}(\Omega)}.
\end{align*}
These estimates finish the proof.
\end{proof}

Now, for $u\in H^{m}(\Omega)$, we define $D_A(u)$ and $m_q(u)$, for any $v\in H^{m}_0(\Omega)$ by
$$
\<D_A(u),v\>_{\Omega}=B_A(u,v)\quad\text{and}\quad \<m_q(u),v\>_{\Omega}=b_q(u,v),
$$

The following result, which is an immediate corollary of Proposition~\ref{m_q and D_A are bounded maps on Omega}, implies that $D_A$, $m_q$ are bounded operators from $H^m(\Omega)$ to $H^{-m}(\Omega)$. The norm on $H^{-m}(\Omega)$ is the usual dual norm given by
$$
\|v\|_{H^{-m}(\Omega)}=\sup_{0\neq\phi\in H^m_0(\Omega)}\frac{|\<v,\phi\>_{\Omega}|}{\|\phi\|_{H^m(\Omega)}}.
$$
\begin{Corollary}\label{m_q and D_A are well-defined on HOmega}
The operators $B_A$ and $b_q$ are bounded from $H^m(\Omega)$ to $H^{-m}(\Omega)$ and satisfy
$$
\|m_q(u)\|_{H^{-m}(\Omega)}\le C\|q\|_{W^{-\frac{m}{2},\frac{2n}{m}}(\R^n)}\|u\|_{H^{m}(\Omega)}
$$
and
$$
\|D_A(u)\|_{H^{-m}(\Omega)}\le C\|A\|_{W^{-\frac{m-2}{2},\frac{2n}{m}}(\R^n)}\|u\|_{H^{m}(\Omega)}
$$
for all $u\in H^{m}(\Omega)$.
\end{Corollary}

Finally, we record and give the proof of the following useful identities.
\begin{Proposition}\label{adjoints of D_A and m_q}
For any $u,v\in H^{m}(\Omega)$, the forms $B_A$ and $b_q$ satisfy the following identities
$$
B_A(u,v)=-B_A(v,u)-b_{D\cdot A}(u,v)\quad\text{and}\quad b_q(u,v)=b_q(v,u).
$$
\end{Proposition}
\begin{proof}
According to the definitions \eqref{def of B_A b_q on Omega} and density of $\mathcal S(\R^n)$ in $H^{m}(\R^n)$, it is sufficient to prove for the case $u,v\in\mathcal S(\R^n)$. This follows by straightforward computations
$$
b_q^{\R^n}(u,v)=\<q,uv\>_{\R^n}=\<m_q(v),u\>_{\R^n},
$$
and using product rule
\begin{align*}
B_A^{\R^n}(u,v)&=\<A,v Du\>_{\R^n}=-\<A,uDv\>_{\R^n}+\<A,D(uv)\>_{\R^n}\\
&=-B_A^{\R^n}(v,u)-\<D\cdot A,uv\>_{\R^n}\\
&=-B_A^{\R^n}(v,u)-b_{D\cdot A}(u,v).
\end{align*}
The proof is thus finished.
\end{proof}

\section{Well-posedness and Dirichlet-to-Neumann map}\label{appendix b}
Let $\Omega\subset \R^n$, $n\ge 3$, be a bounded open set with $C^\infty$ boundary, and let $A\in W^{-\frac{m-2}{2},\frac{2n}{m}}(\R^n,\C^n)\cap \mathcal E'(\overline{\Omega},\C^n)$ and $q\in W^{-\frac{m}{2},\frac{2n}{m}}(\R^n,\C)\cap \mathcal E'(\overline{\Omega},\C)$, where $n>m$.

For $f=(f_0,\dots,f_{m-1})\in \prod_{j=0}^{m-1}H^{m-j-1/2}(\p\Omega)$, consider the Dirichlet problem
\begin{equation}\label{DP appendix}
\begin{aligned}
\mathcal L_{A,q}u&=0\quad\text{in}\quad\Omega,\\
\gamma u&=f\quad\text{on}\quad\p\Omega.
\end{aligned}
\end{equation}
Here, by $\gamma$ we denote the Dirichlet trace operator, given by
$$
\gamma:H^m(\Omega)\to\prod_{j=0}^{m-1}H^{m-j-1/2}(\p\Omega),\quad \gamma u=(u|_{\p \Omega},\p_\nu u|_{\p\Omega},\dots, \p_\nu^{m-1}u|_{\p\Omega}),
$$
which is bounded and surjective, see \cite[Theorem~9.5, page 226]{Gr}.
\medskip

First aim of this appendix is to use the standard variational arguments to show the well-posedness of the problem~\eqref{DP appendix}. First, consider the following inhomogeneous problem
\begin{equation}\label{DP inhomogeneous}
\begin{aligned}
\mathcal L_{A,q}u&=F\quad\text{in}\quad\Omega,\\
\gamma u&=0\quad\text{on}\quad\p\Omega,
\end{aligned}
\end{equation}
with $u\in H^{-m}(\Omega)$.
\medskip

To define a sesquilinear form $a$, associated to the problem~\eqref{DP inhomogeneous}, for $u,v\in C^\infty_0(\Omega)$, we integrate by parts and get
$$
\<\mathcal L_{A,q}u,\overline{v}\>_\Omega=\sum_{|\alpha|=m}\frac{m!}{\alpha!}(D^\alpha u,\overline{D^\alpha v})_{L^2(\Omega)}+\<D_A(u),\overline{v}\>_\Omega+\<m_q(u),\overline{v}\>_\Omega:=a(u,v).
$$
Therefore, $a$ is defined on $H^m(\Omega)$ by
$$
a(u,v):=\sum_{|\alpha|=m}\frac{m!}{\alpha!}(D^\alpha u,\overline{D^\alpha v})_{L^2(\Omega)}+B_A(u,\overline{v})+b_q(u,\overline{v}),\quad u,v\in H^m(\Omega).
$$
Note that this is not a unique way to define a sesquilinear form associated to the problem~\eqref{DP inhomogeneous}.

Now, we show that the sesquilinear form $a$ can be extended to a bounded form on $H^m_0(\Omega)$. Using duality and Proposition~\ref{m_q and D_A are bounded maps}, for $u,v\in H^m_0(\Omega)$, we obtain
\begin{equation}
\begin{aligned}\label{form a is bounded}
|a(u,v)|&\le \sum_{|\alpha|=m}\frac{m!}{\alpha!}\|D^\alpha u\|_{L^2(\Omega)}\|D^\alpha v\|_{L^2(\Omega)}\\
&\qquad+\(\|A\|_{W^{-\frac{m-2}{2},\frac{2n}{m}}(\R^n)}+\|q\|_{W^{-\frac{m}{2},\frac{2n}{m}}(\R^n)}\)\|u\|_{H^{\frac m2}(\Omega)}\|v\|_{H^{\frac m2}(\Omega)}\\
&\le C \|u\|_{H^m(\Omega)}\|v\|_{H^m(\Omega)}.
\end{aligned}
\end{equation}
Thus, the sesquilinear form $a$ is a bounded form on $H^m_0(\Omega)$.
\medskip

Applying Poincar\'e's inequality, we have
\begin{equation}\label{m-poincare inequality}
\|u\|_{H^m(\Omega)}^2\le C\sum_{|\alpha|=m}\|D^\alpha u\|_{L^2(\Omega)}^2,\quad u\in H^m_0(\Omega).
\end{equation}
Write $q=q^\sharp+(q-q^\sharp)$ with $q^\sharp\in L^\infty(\Omega,\C)$ and $\|q-q^\sharp\|_{W^{-\frac{m}{2},\frac{2n}{m}}(\R^n)}$ small enough, and write $A=A^\sharp+(A-A^\sharp)$ with $A^\sharp\in L^\infty(\Omega,\C^n)$ and $\|A-A^\sharp\|_{W^{-\frac{m-2}{2},\frac{2n}{m}}(\R^n)}$ small enough. Using \eqref{m-poincare inequality} and Proposition~\ref{m_q and D_A are bounded maps on Omega}, for $\varepsilon>0$, we obtain that
\begin{align*}
\Re a(u,u)&\ge \sum_{|\alpha|=m}\frac{m!}{\alpha!}\|D^\alpha u\|_{L^2(\Omega)}^2-|B_A(u,u)|-|b_q(u,u)|\\
&\ge C\sum_{|\alpha|=m}\|D^\alpha u\|_{L^2(\Omega)}^2-|B_{A^\sharp}(u,u)|-|b_{q^\sharp}(u,u)|\\
&\quad-|B_{A-A^\sharp}(u,u)|-|b_{q-q^\sharp}(u,u)|\\
&\ge C\|u\|_{H^m(\Omega)}^2-\|A^\sharp\|_{L^\infty(\Omega)}\|Du\|_{L^2(\Omega)}\|u\|_{L^2(\Omega)}-\|q^\sharp\|_{L^\infty(\Omega)}\|u\|_{L^2(\Omega)}^2\\
&\quad-C'\|A-A^\sharp\|_{W^{-\frac{m-2}{2},\frac{2n}{m}}(\R^n)}\|u\|_{H^{\frac m2}(\Omega)}^2-C'\|q-q^\sharp\|_{W^{-\frac{m}{2},\frac{2n}{m}}(\R^n)}\|u\|_{H^{\frac m2}(\Omega)}^2\\
&\ge C\|u\|_{H^m(\Omega)}^2-\|A^\sharp\|_{L^\infty(\Omega)}\frac{\varepsilon}{2}\|Du\|_{L^2(\Omega)}^2-\|A^\sharp\|_{L^\infty(\Omega)}\frac{1}{2\varepsilon}\|u\|_{L^2(\Omega)}^2\\
&\quad-\|q^\sharp\|_{L^\infty(\Omega)}\|u\|_{L^2(\Omega)}^2-C'\|A-A^\sharp\|_{W^{-\frac{m-2}{2},\frac{2n}{m}}(\R^n)}\|u\|_{H^{\frac m2}(\Omega)}^2\\
&\quad-C'\|q-q^\sharp\|_{W^{-\frac{m}{2},\frac{2n}{m}}(\R^n)}\|u\|_{H^{\frac m2}(\Omega)}^2,\quad C,C'>0,\quad u\in H^m_0(\Omega).
\end{align*}
Taking $\varepsilon>0$ sufficiently small, we get
$$
\Re a(u,u)\ge C\|u\|_{H^m(\Omega)}^2-C_0\|u\|_{L^2(\Omega)},\quad C,C_0>0,\quad u\in H^m_0(\Omega).
$$
Therefore, the form $a$ is coercive on $H^m_0(\Omega)$. As the inclusion map $H^m_0(\Omega)\hookrightarrow L^2(\Omega)$ is compact, the operator
$$
\mathcal L_{A,q}=(-\Delta)^m+D_A+m_q:H^m_0(\Omega)\to H^{-m}(\Omega)=(H^m_0(\Omega))',
$$
is Fredholm operator with zero index; see \cite[Theorem~2.34]{McLean}.
\medskip

Positivity of the operator $\mathcal L_{A,q}+C_0:H^m_0(\Omega)\to H^{-m}(\Omega)$ and an application of the Lax-Milgram lemma implies that $\mathcal L_{A,q}+C_0$ has a bounded inverse. By compact Sobolev embedding $H^m_0(\Omega)\hookrightarrow H^{-m}(\Omega)$ and the Fredholm theorem, the equation \eqref{DP inhomogeneous} has a unique solution $u\in H^m_0(\Omega)$ for any $F\in H^{-m}(\Omega)$ if one is outside a countable set of eigenvalues.
\medskip

Now, consider the Dirichlet problem
\begin{equation}\label{DP2 appendix}
\begin{aligned}
\mathcal L_{A,q}u&=0\quad\text{in}\quad\Omega,\\
\gamma u&=f\quad\text{on}\quad\p\Omega,
\end{aligned}
\end{equation}
with $f=(f_0,\dots,f_{m-1})\in \prod_{j=0}^{m-1}H^{m-j-1/2}(\p\Omega)$. We assume that $0$ is not in the spectrum of $\mathcal L_{A,q}:H^m_0(\Omega)\to H^{-m}(\Omega)$. By \cite[Theorem 9.5, page 226]{Gr}, there is $w\in H^m(\Omega)$ such that $\gamma w=f$. According to Corollary~\ref{m_q and D_A are well-defined on HOmega}, we have $\mathcal L_{A,q}w\in H^{-m}(\Omega)$. Therefore, $u=v+w$, with $v\in H^m_0(\Omega)$ being the unique solution of the equation $\mathcal L_{A,q}v=-\mathcal L_{A,q}w\in H^{-m}(\Omega)$, is the unique solution of the Dirichlet problem \eqref{DP2 appendix}.
\medskip

Under the assumption that $0$ is not in the spectrum of $\mathcal L_{A,q}$, the Dirichlet-to-Neumann map $\mathcal N_{A,q}$ is defined as follows: let $f,h\in \prod_{j=0}^{m-1}H^{m-j-1/2}(\p\Omega)$. Then we set
\begin{equation}\label{eq_7_4}
\<\mathcal N_{A,q}f,\overline h\>_{\p\Omega}:=\sum_{|\alpha|=m}\frac{m!}{\alpha!}(D^\alpha u,\overline{D^\alpha v}_h)_{L^2(\Omega)}+B_A(u,\overline v_h)+b_q(u,\overline v_h),
\end{equation}
where $u$ is the unique solution of the Dirichlet problem \eqref{DP2 appendix} and $v_h\in H^m(\Omega)$ is an extension of $h$, that is $\gamma v_h=h$. In this appendix we show that $\mathcal N_{A,q}$ is a well-defined (i.e. independent of the choice of $v_h$) bounded operator
$$
\mathcal N_{A,q}:\prod_{j=0}^{m-1}H^{m-j-1/2}(\p\Omega)\to\(\prod_{j=0}^{m-1}H^{m-j-1/2}(\p\Omega)\)'=\prod_{j=0}^{m-1}H^{-m+j+1/2}(\p\Omega).
$$

Let us first show that the definition \eqref{eq_7_4} of $\mathcal N_{A,q} f$ is independent of the choice of an extension $v_h$ of $h$. For this, let $v_{h,1},v_{h,2}\in H^m(\Omega)$ be such that $\gamma v_{h,1}=\gamma v_{h,2}=h$. Note that $v_{h,1}-v_{h,2}\in H^m_0(\Omega)$. Then we have to show that
\begin{multline}\label{eq_7_5}
\sum_{|\alpha|=m}\frac{m!}{\alpha!}(D^\alpha u,\overline{D^\alpha (v_{h,1}-v_{h,2})}_h)_{L^2(\Omega)}\\
+\<D_A(u),\overline{(v_{h,1}-v_{h,2})}\>_\Omega+\<m_q(u),\overline{(v_{h,1}-v_{h,2})}\>_\Omega=0.
\end{multline}
For any $w\in C^\infty_0(\Omega)$ and for $u\in H^m(\Omega)$ solution of the Dirichlet problem \eqref{DP2 appendix}, we have
$$
0=\langle \mathcal L_{A,q}u, \bar w \rangle=\sum_{|\alpha|=m}\frac{m!}{\alpha!}(D^\alpha u,\overline{D^\alpha w})_{L^2(\Omega)}+\<D_A(u),\overline w\>_\Omega+\<m_q(u),\overline w\>_\Omega.
$$
Density of $C_0^\infty(\Omega)$ in $H^m_0(\Omega)$ and continuity of the form on $H^m_0(\Omega)$ give \eqref{eq_7_5}. 
\medskip

Now we show that $\mathcal N_{A,q}f$ is a well-defined element of $\prod_{j=0}^{m-1}H^{-m+j+1/2}(\p\Omega)$. From \eqref{form a is bounded}, it follows that 
\begin{multline*}
|\< \mathcal N_{A,q}u, \overline h \>_{\p\Omega}|\le C\|u\|_{H^m(\Omega)}\|v_h\|_{H^m(\Omega)}\\
\le C\|\gamma u\|_{\prod_{j=0}^{m-1}H^{m-j-1/2}(\p \Omega)}\|h\|_{\prod_{j=0}^{m-1}H^{m-j-1/2}(\p \Omega)},
\end{multline*}
where 
$$
\|h\|_{\prod_{j=0}^{m-1}H^{m-j-1/2}(\p \Omega)}=(\|h_0\|_{H^{m-1/2}(\p\Omega)}^2+\cdots+\|h_{m-1}\|_{H^{1/2}(\p\Omega)}^2)^{1/2}
$$
is the product norm on the space $\prod_{j=0}^{m-1}H^{m-j-1/2}(\p \Omega)$. 
Here we have used the fact that the extension operator $\prod_{j=0}^{m-1}H^{m-j-1/2}(\p \Omega)\ni h\mapsto v_h\in H^m(\Omega)$ is bounded, again see \cite[Theorem 9.5, page 226]{Gr}. Hence, we have that $\mathcal{N}_{A,q}f$ belongs to $\big(\prod_{j=0}^{m-1}H^{m-j-1/2}(\p \Omega)\big)'= \prod_{j=0}^{m-1}H^{-m+j+1/2}(\p \Omega)$. The proof given above also shows that
$$
\mathcal N_{A,q}:\prod_{j=0}^{m-1}H^{m-j-1/2}(\p\Omega)\to\(\prod_{j=0}^{m-1}H^{m-j-1/2}(\p\Omega)\)'=\prod_{j=0}^{m-1}H^{-m+j+1/2}(\p\Omega)
$$
is bounded.

\section*{Acknowledgements}
The author thanks his advisor Professor Gunther Uhlmann for all his encouragement and support. Many thanks to Professor Katya Krupchyk whose unvaluable comments substantially improved the main result of the paper. The author gratefully acknowledges Professor Winfried Sickel for sharing his copy of the monograph \cite{RS}. Finally, the author is very grateful to Karthik Iyer for pointing out some mistakes in the previous version of the paper and for suggesting the ways to fix them. Specifically, he pointed out that $D_A$ becomes a bounded operator only for $m<n$ (see the proof of Proposition~\ref{m_q and D_A are bounded maps} for details) and that we need to assume little more regularity of the zeroth order perturbation in order to get uniqueness (for details see the last paragraph of Section~\ref{Proof}). The work of the author was partially supported by the National Science Foundation.

\end{document}